\documentclass[11pt]{article}
\usepackage{amsmath,amsthm,amsfonts,amssymb,amscd, amsxtra, mathrsfs,amssymb,amstext,amscd,amsxtra}
\usepackage{url}
\usepackage[margin=2.3 cm,nohead]{geometry}
\usepackage{color}
\usepackage[ruled]{algorithm2e} 

\usepackage{graphicx,float}
\usepackage{mathrsfs,threeparttable,color,soul,colortbl,hhline,rotating,array,multirow,makecell}
\usepackage[table]{xcolor}
\newcommand{\ds}{\displaystyle}
\newcommand{\R}{\mathbb{R}}
\newcommand{\N}{\mathbb{N}}
\newtheorem{theorem}{Theorem}
\newtheorem{lemma}[theorem]{Lemma}
\newtheorem{definition}{Definition}
\newtheorem{corollary}[theorem]{Corollary}

\newtheorem{remark}{Remark}
\newtheorem{example}{Example}
\usepackage{epstopdf}
\usepackage{algorithmic}

\usepackage{graphicx,float}
\usepackage{mathrsfs,threeparttable,color,soul,colortbl,hhline,rotating,array,multirow,makecell}

\ifpdf
  \DeclareGraphicsExtensions{.eps,.pdf,.png,.jpg}
\else
  \DeclareGraphicsExtensions{.eps}
\fi

\begin{document}
\title{Alternating conditional gradient method for  \\convex feasibility problems  \thanks{The authors was supported in part by  CNPq grants 305158/2014-7 and 302473/2017-3,  FAPEG/PRONEM- 201710267000532 and CAPES.}}
\author{
R. D\'iaz Mill\'an \thanks{School of Information Technology, Deakin University, Melbourne,   Australia,  E-mail: {\tt  rdiazmillan@gmail.com}.}
\and
O.  P. Ferreira  \thanks{Instituto de Matem\'atica e Estat\'istica, Universidade Federal de Goi\'as,  CEP 74001-970 - Goi\^ania, GO, Brazil, E-mail:{\tt  orizon@ufg.br}.}
\and
 L. F. Prudente \thanks{Instituto de Matem\'atica e Estat\'istica, Universidade Federal de Goi\'as,  CEP 74001-970 - Goi\^ania, GO, Brazil, E-mail:{\tt  lfprudente@ufg.br}.}
 }

\maketitle

\begin{abstract}
The classical convex feasibility problem in a  finite dimensional Euclidean space is studied in the present paper.  We are interested in two cases. First,  we assume to  know how to compute an exact project onto one of the sets involved and the other set is  compact such that the conditional gradient (CondG) method  can be used for computing efficiently an inexact projection on it. Second,  we assume that both sets involved are compact   such that   the CondG method  can be used  for  computing efficiently   inexact projections on them.  We  combine   alternating projection method with  CondG method   to design a new method,  which  can be seen as an inexact feasible  version  of  alternate projection method.  The proposed method generates two different sequences belonging to each involved set, which converge to a point  in the intersection of them  whenever  it is not empty. If the intersection is empty, then the sequences converge to points in the respective sets whose distance  is equal to the distance between the sets in consideration.
\end{abstract}

{\bf Keywords:}  Convex feasibility  problem, alternating projection method,  conditional gradient method,   inexact projections.

{\bf AMS:} 65K05,   	90C30,  	90C25.

\section{Introduction}
The  {\it classic convex feasibility problem} consists  of finding a point in the intersection of two sets. It is formally state as follows:
\begin{equation}\label{def:InexactProjProb}
\mbox{find } x_*\in A\cap B, 
\end{equation}
where $A, B \subset \mathbb{R}^n$ are convex, closed,  and nonempty sets.  Although we are not concerned with practical issues at this time, we emphasize that several  practical applications appear  modeled  as Problem~\eqref{def:InexactProjProb}; see for example  \cite{Combettes1996, Combettes1999, HesseLukeNeumann2014} and  references therein.   Among the methods to solve Problem~\eqref{def:InexactProjProb}, the  {\it alternating projection method} is one of the most interesting and popular, with a long history dating back to J. von Neumann \cite{Neumann1950}.  Since this seminal work, the  alternating  projection method has attracted the attention of the scientific community working on optimization,  papers  dealing with this method include \cite{BauschkeBorwein1993, Bregman1965, CheneyGoldstein1959}.  Perhaps one of the factors that explains this interest is its simplicity and ease of implementation, making application to large-scale problems very attractive.   Several variants of this method have arisen and several theoretical and practical issues related to it have been discovered over the years, resulting in a wide literature on the subject.  For a historical perspective of this  method; see, for exemple  \cite{BauschkeBorwei1996} and  a complete annotated bibliography of books and review can be found in  \cite{CensoCegielski2015}.  

The aim of this paper  is   present  a new method to solve  Problem~\ref{def:InexactProjProb}. The proposed method is based on  the  alternating  projection method. For design the method, the {\it conditional gradient method  (CondG  method)} also known as {\it  Frank-Wolfe  algorithm} developed by Frank and Wolfe in 1956  \cite{FrankWolfe1956} (see also   \cite{LevitinPoljak1966}) is used to compute feasible inexact projections on the sets in consideration, which will be  named as {\it alternating conditional gradient  (ACondG) method}. We present  two versions of the method.  First,  we  assume that  we  know how to compute an exact project onto one of the sets involved. Besides,  we assume that the other  set   is compact with a special structure such that  CondG method  can be used  for  computing efficiently  feasible inexact projections on it.  Second,  we assume that both sets  are compact  with    special structures  such that   the CondG method  can be used  for  computing  feasible inexact projections on them. The ACondG method  proposed,   generates two sequences $(x^k)_{k\in\mathbb{N}}\subset A$ and $(y^k)_{k\in\mathbb{N}} \subset B$.  The  mains  obtained results are as follows. If  $A\cap B\neq \varnothing$, then the sequences   $(x^k)_{k\in\mathbb{N}}$ and  $(y^k)_{k\in\mathbb{N}}$  converge to a point  $x^*$ belonging  to   $A\cap B$.   If  $A\cap B= \varnothing$, then the sequences  $(x^k)_{k\in\mathbb{N}}$ and  $(y^k)_{k\in\mathbb{N}}$, converge respectively to $x^*\in A$ and $y^*\in B$ satisfying  $\|x^*-y^*\|=dist(A,B)$, where $dist(A,B)$ denotes  the distance between the sets $A$ and $B$. 

From a practical point of view, considering  methods that use inexact projections are particularly interesting. Indeed,   one drawback of  methods that use exact projections is the need  to solve a quadratic problem at each stage,  which  may substantially increasing the cost per iteration  if the number of unknowns is large. In fact, it may not be justified to carry out  exact projections when the iterates   are  {\it far} from the solution of the problem. For this reason, seeking  to make the alternating  projection method  more efficient,  we  use the CondG method to compute    feasible inexact projections  rather than exact ones.  It is noteworthy that the CondG method is easy to implement, has low computational cost per iteration, and readily exploits separability and sparsity, resulting in high computational performance in different classes  of compact sets, see \cite{Dunn1980, FreundGrigas2016, GarberHazan2015, Jaggi2013,  LevitinPoljak1966}. Therefore, we believe that all of these features accredit the method as being quite appropriate for our purpose, which has also been used for similar aims  \cite{Clarkson2010,Fukushima1984,MaxJefferson2017, LussTeboulle2013}.   As aforementioned, the proposed  method performs  iterations  alternately on the sets $A$ and $B$ only approximately, but them  become increasingly accurate  in relation to the progress of previous iterations. Therefore, the  resulting method can be seen as an inexact  version of the classical alternate projection method.  It is worth noting that others  approximate projections have been widely used in the literature. For instance, approximate projection can be  performed by projecting onto  the hyperplane separating the set and the  point to be projected, see \cite{Fukushima1986},  and for more examples, see \cite{BauschkeCombettes2001, CegielskiReichZalas2018, Combettes2001, DiazMillan2019,    DrusvyatskiyLewis2019, Iusem2019}.  However,   inexact projections obtained in this way are infeasible to the set to be projected, in contrast with   feasible inexact  projections  propose here. Moreover, as far as we know, the combination of the  conditional method  with the   alternate directions method  for  designing a new method to solve Problem~\eqref{def:InexactProjProb}  has not yet been considered.

The organization of the paper is as follows. In section \ref{sec:Preliminares}, we present some notation and basic results used in our presentation. In section \ref{Sec:CondG} we describe the conditional gradient  method and present some results related to it.   In  sections~\ref{Sec:ACondG-1} and  we present, respectively,  the first and second version of inexact  alternating  projection method  to solve Problem~\eqref{def:InexactProjProb}.   Some numerical experiments are provided in section \ref{Sec:NumExp}.   We conclude the paper with some remarks in section \ref{Sec:Conclusions}.

\bigskip

\noindent {\bf Notation.} We denote: ${\mathbb{N}}=\{0,1,2,\ldots\}$, $\langle \cdot,\cdot \rangle$ is the usual inner product, $\|\cdot\|$ is the Euclidean norm, and $[v]_i$ is the $i$-th component of the vector $v$.

\section{{Preliminaries}} \label{sec:Preliminares}
In this section, we present  some preliminary  results used throughout the paper.  The {\it projection  onto a  closed  convex set $C\subset  \mathbb{R}^n $} is the mapping   $P_C:   \mathbb{R}^n \to C$ defined      by
$$
P_C(v):=  \arg\min_{z \in  C}\|v-z\|.
$$
In the  next lemma we present some important properties of the  projection mapping. 
\begin{lemma}\label{le:projeccion}
Let  $C\subset \mathbb{R}^n$ be any nonempty closed and convex set and $P_C$ the projection mapping onto $C$. For all $v\in  \mathbb{R}^n$, the following properties hold:
\begin{enumerate}
\item[(i)]$\langle v-P_C(v), z-P_C(v)\rangle\leq0$, for all $z\in C$;
\item[(ii)] $\|{P}_C(v)-z\|^2\leq \|v-z\|^2- \|{P}_C(v)-v\|^2$, for all $z\in C$;  
\item[(iii)]  the projection mapping $P_C$ is continuous.
\end{enumerate}
\end{lemma}
\begin{proof}
The items (i) and (iii)  are proved in  \cite[Proposition 3.10, Theorem 3.14]{BauschkeCombettes20011}. For item (ii),  combine  $ \|v-z\|^2=\|{P}_C(v)-v\|^2+\|{P}_C(v)-z\|^2-2\langle {P}_C(v)-v, {P}_C(v)-z\rangle$ with  item (i). 
\end{proof}
Let  $C,D \subset \mathbb{R}^n$ be  convex, closed,  and nonempty sets.  Define the {\it distance    between the sets $C$ and $D$}   by $dist(C,D):=\inf\{ \|v-w\|: ~v\in C, w\in D\}$. 
\begin{lemma} \label{le:ppseq}
Let   $C$ be a  compact and convex set and $D$ be  a  closed  and  convex set.  Assume that the sequences    $(v^k)_{k\in\mathbb{N}}\subset  C$ and $(w^k)_{k\in\mathbb{N}}\subset D$   satisfy  the following two conditions:
\begin{itemize}
\item[(c1)] $\lim_{k\to +\infty}\|v^{k+1}-v^k\|=0$ and  $ \lim_{k\to +\infty}\|w^{k+1}-w^k\|=0$;
\item[(c2)] $\lim_{k\to +\infty} \|w^{k+1}-P_D(v^k)\|=0$  and  $\lim_{k\to +\infty} \|v^{k+1}-P_C(w^{k+1})\|=0$.
\end{itemize}
Then,  each cluster point ${\bar v}$  of $(v^k)_{k\in\mathbb{N}}$ is a fixed point of $P_CP_D$, i.e.,  ${\bar v}=P_CP_D({\bar v})$. Moreover, $\lim_{k\to \infty}\|v^k-w^k\|=dist(C,D)$ and  $\lim_{k\to \infty}(v^k-w^k)=P_{C-D}(0)$, where $P_{C-D}(0):=\arg\min_{v \in  C, w\in D}\|v-w\|$.
\end{lemma}
\begin{proof}
Let ${\bar v}$ be a cluster point of $(v^k)_{k\in\mathbb{N}}$.  Let  $(v^{k_j})_{j\in\mathbb{N}}$ be  a subsequence of $(v^k)_{k\in\mathbb{N}}$ with $\lim_{{j} \to +\infty}v^{k_j}={\bar v}$ and consider $(w^{k_j+1})_{j\in\mathbb{N}}$ which is a subsequence of  $(w^{k+1})_{k\in\mathbb{N}}$.  Since $C$ is   compact, $(v^k)_{k\in\mathbb{N}}\subset C$ and $\lim_{k\to +\infty} \|w^{k+1}-P_D(v^k)\|=0$, we conclude that  $(w^{k_j+1})_{j\in\mathbb{N}}$  is bounded. Thus, there exists a cluster point ${\bar w}$ of $(w^{k_j+1})_{j\in\mathbb{N}}$ and a subsequence $(w^{k_\ell+1})_{\ell\in\mathbb{N}}$  of $(w^{k_j+1})_{j\in\mathbb{N}}$ with $\lim_{{\ell} \to +\infty}w^{k_\ell+1}={\bar w}$. Futhermore, the corresponding subsequence $(v^{k_\ell})_{j\in\mathbb{N}}$ of $(v^{k_j})_{j\in\mathbb{N}}$   also satisfies   $\lim_{{\ell} \to +\infty}v^{k_\ell}={\bar v}$. Due to $(\|v^{k+1}-v^k\|^2)_{k\in\mathbb{N}}$ and $(\|w^{k+1}-w^k\|^2)_{k\in\mathbb{N}}$ converge  to zero,  we also have  $\lim_{{\ell} \to +\infty}v^{k_\ell+1}={\bar v}$ and $\lim_{{\ell} \to +\infty}w^{k_\ell}={\bar w}$. Moreover, it follows from Lemma~\ref{le:projeccion} (ii) and conditions  (c1) and (c2)  that
$$
{\bar w}=P_D({\bar v}), \qquad  {\bar v}=P_C({\bar w}), 
$$
respectively.  Hence, last equalities  imply  ${\bar v}=P_CP_D({\bar v})$ and,  by using  \cite[Theorem 2]{CheneyGoldstein1959},    we have  $\|{\bar v}-P_D({\bar v})\|=dist(C,D)$.  Therefore, we conclude that  each cluster point ${\bar v}$  of $(v^k)_{k\in\mathbb{N}}$ is a fixed point of $P_CP_D$, i.e.,  ${\bar v}=P_CP_D({\bar v})$, which proves the first statement.  Moreover, for each  subsequence $(v^{k_\ell})_{\ell\in\mathbb{N}}$ of  $(v^k)_{k\in\mathbb{N}}$ such that  $\lim_{{\ell} \to +\infty}v^{k_\ell}={\bar v}$, there exists a subsequence $(w^{k_\ell})_{\ell\in\mathbb{N}}$  of $(w^k)_{k\in\mathbb{N}}$ with $\lim_{{\ell} \to +\infty}w^{k_\ell}=P_D({\bar v})$ and $\|{\bar v}-P_D({\bar v})\|=dist(C,D)$.  Consequently, the sequence  $(\|v^k-w^k\|)_{k\in\mathbb{N}}$  converges to the distance between $A$ and $B$, i.e., $\lim_{k\to \infty}\|v^k-w^k\|=dist(C,D)$  and this proves   the second statement. Considering that  $D$ is a closed  convex set and $C$ is a compact convex set, it follows  that   $C-D$ is also a closed and convex  set, which implies   $P_{C-D}(0)$ is a singleton. Therefore, we obtain that   $(v^k-w^k)_{k\in\mathbb{N}}$ converges to $P_{A-B}(0)$, i.e., $\lim_{k\to \infty}(v^k-w^k)=P_{C-D}(0)$ (see,  \cite[Lemma~2.3]{BauschkeBorwein1994}), concluding the proof of the lemma. 
\end{proof}
\begin{definition}
Let $S$ be a nonempty subset of $\mathbb{R}^n$. A sequence $(v^k)_{k\in\mathbb{N}}\subset \mathbb{R}^n$ is said to be quasi-Fej\'er convergent to $S$, if and only if, for all $v\in S$ there exists $k_0\ge 0$ and a summable sequence $(\epsilon_k)_{k\in\mathbb{N}}$, such that $\|v^{k+1}-v\|^2 \le \|v^k - v\|^2+\epsilon_k$ for all $k\ge k_0$. If for all $k\in\mathbb{N}$, $\epsilon_k=0$, the sequence $(v^k)_{k\in\mathbb{N}}$ is said to be Fej\'er convergent to $S$.
\end{definition}
\begin{lemma}\label{fejer}
Let  $(v^k)_{k\in\mathbb{N}}$  be   quasi-Fej\'er convergent to $S$.  Then, the following conditions hold:
\begin{itemize}
\item[(i)] the sequence $(v^k)_{k\in\mathbb{N}}$ is bounded;
\item[(ii)] for all  $v\in S$, the sequence $(\|v^k-v\|)_{k\in\mathbb{N}}$ is convergent.
\item[(iii)] if a cluster point ${\bar v}$ of $(v^k)_{k\in\mathbb{N}}$  belongs to $S$, then  $(v^k)_{k\in\mathbb{N}}$ converges to ${\bar v}$.
\end{itemize}
\end{lemma}
\section{Conditional gradient (CondG) method} \label{Sec:CondG}
In the following  we remind the classical {\it conditional gradient method (CondG$_C$)} to compute feasible inexact projections with respect to a compact  convex set $C$ and some results  related to it. We also prove two   important inequalities  related to this method  that will be useful   to establish    our main results.  For presenting the method, we assume the existence of a linear optimization oracle (or simply LO oracle) capable of minimizing linear functions over the constraint set $C$.  We formally state the  CondG$_{C}$ method to calculate an inexact projection of $v\in {\mathbb R}^n$ with respect $u\in C$,  with the following   input data:   {\it a relative  error tolerance function   $\varphi: {\mathbb R}^n\times {\mathbb R}^n\times {\mathbb R}^n \to {\mathbb R}_{+}$ satisfying the following inequality 
\begin{equation} \label{eq:pfphi}
\varphi_{ \gamma, \theta, \lambda}(u, v, w)\leq \gamma \|v-u\|^2 +\theta  \|w-v\|^2 +   \lambda \|w-u\|^2 , \qquad  \forall~ u, v, w\in  {\mathbb R}^n, 
\end{equation}
 where  $  \gamma, \theta, \lambda \geq 0$  are given forcing parameters}. 
\begin{algorithm}[htb] 
\begin{description}
\item[Input: ]  Take $ \gamma, \theta, \lambda \in {\mathbb R}_+$,  $v\in {\mathbb R}^n$,  $u, w\in C$, and     $\varphi_{ \gamma, \theta, \lambda}$. Set $w_0=w$ and  $\ell=0$.
\item[ Step 1.] Use a LO oracle to compute an optimal solution $z_\ell$ and the optimal value $s_{\ell}^*$ as
\begin{equation}\label{eq:CondG_{C}$}
z_\ell := \arg\min_{z \in  C} \,\langle w_\ell-v, ~z-w_\ell\rangle,  \qquad s_{\ell}^*:=\langle  w_\ell-v, ~z_\ell-w_\ell \rangle.
\end{equation}
\item[ Step 2.] If $-s^*_{\ell}\leq  \varphi_{ \gamma, \theta, \lambda} (u, v, w_\ell)$, then {\bf stop}.  Set $w^+:=w_\ell$.
\item[ Step 3.] Compute $\alpha_\ell \in \, (0,1]$ and $w_{\ell+1}$ as
\begin{equation}\label{eq:stepsize}
w_{\ell+1}:=w_\ell+ \alpha_\ell(z_\ell-w_\ell), \qquad {\alpha}_\ell: =\min\left\{1, \frac{-s^*_{\ell}}{\|z_\ell-w_\ell\|^2}  \right\}.
\end{equation}
\item[ Step 4.] Set $\ell\gets \ell+1$, and go to step {\bf  1}.
\item[Output:]   $w^+:=w_\ell$.
\end{description}
\caption{{\bf CondG$_{C}$ method} $w^+$:=\mbox{CondG$_{C}$}\,$(\varphi_{ \gamma, \theta, \lambda}, u,v)$}
 \label{Alg:CondG}
\end{algorithm}

Let us  describe the main features of  CondG$_{C}$ method; for further details,  see, for example,  \cite{BeckTeboulle2004, Jaggi2013, LanZhou2016}. Let  $v\in {\mathbb R}^n$,    $\psi: \mathbb{R}^n \to \mathbb{R}$ be defined by $\psi(z) := \|z - v\|^2/2$,  and  $C\subset  {\mathbb R}^n$ a convex compact set. It is worth mentioning that the above  \emph{CondG$_{C}$  method}  can be viewed as a specialized version of the classic conditional gradient method  applied to the problem $\min_{z \in C}\psi(z)$. In this case,  \eqref{eq:CondG_{C}$} is equivalent to $s_{\ell}^*:=\min_{z \in C}\langle \psi'(w_\ell) ,~z-w_\ell\rangle$. Since $\psi$
is convex we have
$$
\psi(z)\geq \psi(w_\ell) + \langle \psi'(w_\ell) ,~z-w_\ell\rangle\geq    \psi(w_\ell)  +   s_{\ell}^*, \qquad \forall~z\in C.
$$
Set  $ w_*:=\arg \min_{z \in C}\psi(z)$ and  $\psi^*:= \min_{z \in C}\psi(z)$. Letting $z= x_*$ in the last inequality  we obtain that   $\psi(w_\ell)\geq \psi^* \geq \psi(w_\ell)  +   s_{\ell}^*$, which implies that $s_{\ell}^*\leq 0$. Thus, we conclude that 
$$
-s_{\ell}^*=\langle  v-w_\ell, ~z_\ell-w_\ell \rangle \geq 0\geq  \langle  v-w_*, ~z-w_* \rangle, \qquad \forall~z\in C.
$$
Therefore, we state the  stopping criteria as $-s_{\ell}^*\leq \varphi_{ \gamma, \theta, \lambda}(u, v, w_\ell)$. Moreover, if the  $\emph{CondG$_{C}$}$   method computes  $w_\ell \in C$ satisfying $-s_{\ell}^*\leq  \varphi_{ \gamma, \theta, \lambda}(u, v, w_\ell) $, then the method terminates. Otherwise, it computes the stepsize $\alpha_\ell = \arg\min_{\alpha \in [0,1]} \psi(w_\ell + \alpha(z_\ell - w_\ell))$  using exact minimization.  Since $z_\ell$, $w_\ell \in C$  and $C$ is convex, we conclude from  \eqref{eq:stepsize}  that $w_{\ell+1} \in C$, thus the   $\emph{CondG$_{C}$}$ method generates a sequence in $C$.  Finally,   \eqref{eq:CondG_{C}$} implies that  $\langle  v-w_\ell, ~z-w_\ell\rangle\leq -s_{\ell}^*$, for all  $ z\in C$.  Hence, considering  the  stopping criteria     $-s_{\ell}^*\leq \varphi_{ \gamma, \theta, \lambda}(u, v, w_\ell)$, we conclude that {\it the output of $\mbox{CondG$_{C}$}$ method is  a feasible inexact projection $w^+=\mbox{CondG$_{C}$}\,(\varphi_{ \gamma, \theta, \lambda}, u,v)$ of the point $v\in {\mathbb R}^n$ with respect $u\in C$ onto $C$}, i.e.,   
\begin{equation}\label{inexactset}
\langle  v-w^+, ~z-w^+\rangle\leq   \varphi_{ \gamma, \theta, \lambda}(u, v, w^+), \qquad \forall~z\in C.
\end{equation}
\begin{remark}
Since $C\subset {\mathbb R}^n$ is a closed convex set, then by   Lemma~\ref{le:projeccion} (i)  together with \eqref{inexactset}  we obtain, for any $v\in {\mathbb R}^n$ and $u\in C$, that ${P}_C(v)=\mbox{CondG$_{C}$}\,(\varphi_{ 0, 0,  0}, u,v)$, where ${P}_C(v)$ denotes the exact projection of $v$  onto $C$ and $\varphi_{ 0, 0,  0}(u, v, w)\equiv 0$.
\end{remark}
The following two theorem state   well-known convergence rate for  classic conditional gradient method applied to problem $\min_{z \in C}\psi(z)$, see \cite{GarberHazan2015, Jaggi2013, LevitinPoljak1966}.  Let us first  remind some basic  properties of the function $\psi$ over the set $C$: 
\begin{itemize} 
\item[(i)]$\psi(z)=\psi(w) + \langle \psi'(w) ,~z-w\rangle +\|z-w\|^2/2$, for all $z, w\in C$;
\item[(ii)]  $\psi(z)-\psi(w_*)\geq \|z-w_*\|^2/2$, for all $z\in C$;
\item[(iii)] $\|\psi'(z)\|\geq\psi(w_*)= d(v, C)$, for all $z\in C$.
\end{itemize} 
Since $C$ is a compact  set, we define  the diameter of $C$ by $d_{C}:=\max_{z, w\in C}\|z-w\|$. The statement of the first convergence  result is as follows.
\begin{theorem}  \label{th:fcr}
Let $\{w_{\ell}\}$  be the sequence generated by Algorithm~\ref{Alg:CondG}. Then,  $\psi(w_\ell) -\psi(w_*)\leq (8d_{C}^2)/\ell$, for all $\ell\geq 1$.  Consequently, (by using item (ii) above)  we have  $\|w_\ell -w_*\|\leq 4d_{C}/\ell$, for all $\ell\geq 1$.
\end{theorem}
The rate of convergence in Theorem~\ref{th:fcr} is improved for $\alpha_C$-strongly convex.  We say that a convex set $C$  is {\it $\alpha_C$-strongly convex} if,  for any $z, w\in C$, $t\in [0, 1]$ and any vector $u\in {\mathbb R}^n$, it holds that  $tz+(1-t)w+(\alpha_C/2)t(1-t)\|z-w\|^2u\in C$.  
\begin{theorem}  \label{eq:rccgm2}
Assume that $C$ is  a $\alpha_C$-strongly convex set.  Let $\{w_{\ell}\}$  be the sequence generated by Algorithm~\ref{Alg:CondG} and set  $q:=\max\{1/2, 1-\alpha_Cd(v,C)/8\}<1$. Then,  $\psi(w_{\ell+1}) -\psi(w_*)\leq q\left(\psi(w_{\ell}) -\psi(w_*)\right)$, for all $\ell\geq 1$.  Consequently,  we have an  exponentially convergence rate as follows $\psi(w_\ell) -\psi(w_*)\leq (\psi(w_0) -\psi(w_*))q^\ell$, for all $\ell\geq 1$. Furthermore,  (by using item (ii) above)  we have  $\|w_\ell -w_*\|\leq  (\psi(w_0) -\psi(w_*))q^\ell$, for all $\ell\geq 1$.
\end{theorem}
Let us  present   two  useful  properties  of $\emph{CondG$_{C}$}$ method  that will  play important roles  in the remainder of this paper.
\begin{lemma} \label{pr:condi}
Let $v \in {\mathbb R}^n$, $u \in C$,  $\gamma, \theta, \lambda \geq 0$ and   $w^+=\mbox{CondG$_{C}$}\,(\varphi_{ \gamma, \theta, \lambda}, u,v)$. Then,  there holds
\begin{equation}  \label{eq:fep1}
 \|w^+-z\|^2\leq \|v-z\|^2+ \frac{2\gamma+2\lambda}{1-2\lambda}\|v-u\|^2- \frac{1-2\theta}{1-2\lambda} \|w^+ -v\|^2, \qquad z\in C, 
\end{equation}
for $0\leq \lambda<1/2$.   Consequently, if $z=P_C(v)$ then 
\begin{equation}  \label{eq:fep2}
\|w^+-P_C(v)\|^2\leq  \frac{2\gamma+2\lambda}{1-2\lambda}\|v-u\|^2+\frac{2\theta}{1-2\lambda} \|w^+-v\|^2.
\end{equation}
\end{lemma}
\begin{proof}
First, note that $  \|w^+ -z\|^2 = \|v -z\|^2 - \|v-w^+\|^2+2\langle v-w^+,z-w^+ \rangle$,  for all $z  \in C$.  Since $w^+=\mbox{CondG$_{C}$}\,(\varphi_{ \gamma, \theta, \lambda}, u,v)$, combining the last  inequality  with \eqref{inexactset} and \eqref{eq:pfphi}, after some algebraic manipulation,  we  obtain
\begin{equation} \label{eq:fg}
\|w^+-z\|^2\leq \|v-z\|^2+2 \gamma\|v-u\|^2-(1-2\theta)\|v-w^+\|^2+2\lambda\|w^+-u\|^2, \qquad z\in C.
\end{equation}
On the other hand,    $ \|w^+-u\|^2=\|w^+-v\|^2+\|u-v\|^2-2\langle w^+-v,u-v\rangle$, which  implies that  
$$
 \|w^+-u\|^2= \|u-v\|^2- \|w^+-v\|^2 +2 \langle v-w^+,u-w^+\rangle.
$$
Since $w^+=\mbox{CondG$_{C}$}\,(\varphi_{ \gamma, \theta, \lambda}, u,v)$ and $u \in C$, using  \eqref{inexactset} with $z=u$ and \eqref{eq:pfphi} with $w=w^+$, after some calculations,  the last equation  implies 
$$
\|w^+-u\|^2\leq \frac{1+2\gamma}{1-2\lambda}\|v-u\|^2- \frac{1-2\theta}{1-2\lambda} \|w^+ -v\|^2.
$$
Combining last inequality with \eqref{eq:fg} we obtain   \eqref{eq:fep1}. We proceed  to prove   \eqref{eq:fep2}. First note that  letting $z=P_B(v)$ into   inequality  \eqref{eq:fep1},  the resulting inequality can be  equivalently rewriting as follows 
$$
\|w^+-P_C(v)\|^2\leq \|v-P_C(v)\|^2-\|w^+-v\|^2 + \frac{2\gamma+2\lambda}{1-2\lambda}\|v-u\|^2- \frac{2\lambda-2\theta}{1-2\lambda} \|w^+-v\|^2.
$$
Thus, considering that  $\|v-P_C(v)\|^2-\|w^+-v\|^2\leq 0$ and  $0\leq \lambda<1/2$  the desired inequality  follows, which concludes  the  proof. 
\end{proof}
\begin{corollary} \label{cr:condi}
Assume that  set   $E:=\{z\in C: \|z-P_{D}(z)\|=dist(C,D)\} \neq  \varnothing$. Let  $v \in D$, $u \in C$,  $\gamma, \theta, \lambda \geq 0$,  $w^+=\mbox{CondG$_{C}$}\,(\varphi_{ \gamma, \theta, \lambda}, u,v)$. Then,   for each ${\bar z} \in E$,  there holds
\begin{multline*} 
\|w^+-{\bar z}\|^2 \leq \|u-{\bar z}\|^2 + 2 \langle u-v, P_D(\bar{z})-v\rangle +\frac{2\gamma +2\theta}{1-2\lambda}\|v-u\|^2- \frac{2\lambda-2\theta}{1-2\lambda} \|w^+-v\|^2. 
\end{multline*}
for $0\leq \lambda<1/2$. 
\end{corollary}
\begin{proof}
Applying  \eqref{eq:fep1} of Lemma \ref{pr:condi} with $z={\bar z}$ we obtain 
$$
\|w^+-{\bar z}\|^2\leq \|v-{\bar z}\|^2 +\frac{2\gamma+2\lambda}{1-2\lambda}\|v-u\|^2- \frac{1-2\theta}{1-2\lambda} \|w^+-v\|^2, 
$$
which is  equivalently to 
\begin{equation} \label{eq:stig}
\|w^+-{\bar z}\|^2 \leq  \|v-{\bar z}\|^2 -\|w^+-v\|^2+\frac{2\gamma+2\theta}{1-2\lambda}\|v-u\|^2- \frac{2\lambda-2\theta}{1-2\lambda} \|w^+-v\|^2. 
\end{equation}
First,  note that     $\|u-v\|^2+\|v-{\bar z}\|^2 = \|u-{\bar z}\|^2 +2\langle u-v, {\bar z}-v\rangle$. It turns out that $\langle u-v, {\bar z}-v\rangle= \langle u-v, {\bar z}-P_D(\bar{z})\rangle + \langle u-v, P_D({\bar z})-v\rangle$. Hence,    from these two equalities, we obtain
\begin{equation} \label{eq:filt}
 \|u-v\|^2+\|v-{\bar z}\|^2 = \|u-{\bar z}\|^2 +2\langle u-v, {\bar z}-P_D(\bar{z})\rangle + 2\langle u-v, P_D({\bar z})-v\rangle.
 \end{equation}
 Due to $\bar{z}\in E$,  we have  $\|\bar{z}-P_B(\bar{z})\|=dist(C,D)\leq \|w^+-v\|$ and,   by Cauchy-Schwartz inequality,     $2\langle u-v, \bar{z}-P_D(\bar{z})\rangle\leq 2\|u-v\| \|\bar{z}-P_D(\bar{z})\|\leq \|u-v\|^2+\|\bar{z}-P_D(\bar{z})\|^2$. These inequalities imply $2\langle u-v, \bar{z}-P_D(\bar{z})\rangle\leq \|u-v\|^2  +   \|w^+-v\|^2 $, which combined  with \eqref{eq:filt} yields
 $$
 \|u-v\|^2+\|v-{\bar z}\|^2 \leq  \|u-{\bar z}\|^2 +  \|u-v\|^2  +   \|w^+-v\|^2 +2\langle u-v, P_D({\bar z})-v\rangle. 
 $$
This inequality is equivalente to $\|v-{\bar z}\|^2-  \|w^+-v\|^2  \leq  \|u-{\bar z}\|^2  +2\langle u-v, P_D({\bar z})-v\rangle$, which together  with \eqref{eq:stig} yields the desired result. 
\end{proof}
Let us end this section by presenting some examples of functions satisfying   \eqref{eq:pfphi}.
\begin{example}
The  functions  $\varphi_1, \varphi_3, \varphi_3,  \varphi_4,  \varphi_5,   : {\mathbb R}^n\times {\mathbb R}^n\times {\mathbb R}^n \to {\mathbb R}_{+}$   defined by 
$\varphi_1(u, v, w):=\gamma\|v-u\|^2$, $\varphi_2(u, v, w):= \theta\|w-v\|^2$,   $\varphi_3(u, v, w):= \gamma \theta\|v-u\|\|w-v\|$ and $\varphi_5(u, v, w)= \gamma \|v-u\|^2 +\theta  \|w-v\|^2 + \lambda \|w-u\|^2 $ satisfy \eqref{eq:pfphi}.
\end{example}
\section{The ACondG method with  inexact  projection onto one set} \label{Sec:ACondG-1}
Next we present our  first version of inexact  alternating  projection method  to solve Problem~\eqref{def:InexactProjProb}, by using the CondG method to compute feasible inexact projections with respect to  one of the sets in consideration, which will be  named as {\it alternating conditional gradient-1 (ACondG-1)} method. For that,  we assume that to  find  an exact project onto the convex set $B$, not necessarily  compact,  is an easy task. We also assume that  $A$ is a  convex and compact set and the projection onto it  can be approximate  by using the CondG$_{A}$ method. In this case,  the  {\it ACondG-1 method with feasible inexact projections onto one of the sets},   for solving  the classic feasibility Problem~\eqref{def:InexactProjProb},  is formally defined as follows:
\begin{algorithm} 
\begin{description}
\item[ Step 0.] Let   $(\lambda_k)_{k\in\mathbb{N}}$, $(\gamma_k)_{k\in\mathbb{N}}$ and $(\theta_k)_{k\in\mathbb{N}}$  be  sequences of nonnegative real numbers and the associated function $\varphi_k:=\varphi_{ \gamma_k, \theta_k, \lambda_k}$, as defined in \eqref{eq:pfphi}. Let $x_0\in A$.    If $x^0\in B$, then {\bf stop}. Otherwise, initialize $k\leftarrow 0$.
\item[ Step 1.] Compute $P_B(x^k)$  and  set the next iterate  $y^{k+1}$ as  follows 
\begin{equation} \label{eq:P_A(x^k)}
y^{k+1}:=P_B(x^k).
\end{equation}
If   $y_{k+1}\in A$, then {\bf stop}. 
\item[ Step 2.] Use   Agorithm \ref{Alg:CondG} to   compute $\mbox{CondG$_{A}$}\,(\varphi_k, x^k,y^{k+1})$   and set the  iterate  $x^{k+1}$ as  follows  
\begin{equation} \label{eq:P_B(y^k)}
x^{k+1}:=\mbox{CondG$_{A}$}\,(\varphi_k, x^k,y^{k+1}).
\end{equation}
If   $x_{k+1}\in B$, then {\bf stop}. 
\item[ Step 3.]  Set $k\gets k+1$, and go to \textbf{Step~1}.
\end{description}
\caption{CondG method with  inexact  projection onto one set (ACondG-1)}
\label{Alg:ACondG-1}
\end{algorithm}

First of all note that $x^k\in A$ and  $y^k\in B$. Thus, if Algorithm~\ref{Alg:ACondG-1} stops, it means that a point belonging to $A\cap B$ has been found. Therefore,  we assume that   $(x^k)_{k\in\mathbb{N}}$ and $(y^k)_{k\in\mathbb{N}}$  generated by Algorithm~\ref{Alg:ACondG-1} are  infinity sequences.  In the following we will analyze Algorithm~\ref{Alg:ACondG-1}   first assuming that $A\cap B$ is nonempty   and,   then,  considering that $A\cap B$ is empty.
\subsection{The ACondG-1 method  for two sets with nonempty intersection} 
In this section we assume that  $A\cap B\neq  \varnothing$. To proceed with the convergence analysis of ACondG-1 method we need to {\it assume   that the forcing sequences $(\lambda_k)_{k\in\mathbb{N}}$, $(\gamma_k)_{k\in\mathbb{N}}$ and $(\theta_k)_{k\in\mathbb{N}}$   satisfy }
\begin{equation} \label{eq:fsa}
 \theta_k\leq  \theta< 1/2, \qquad  2\gamma_k+4\lambda_k< \sigma<1  \qquad k=0, 1, \ldots, 
\end{equation}
where $ { \theta}$ and  $ \sigma$ are positive  real constants. 
\begin{theorem}
The sequences   $(x^k)_{k\in\mathbb{N}}$ and  $(y^k)_{k\in\mathbb{N}}$  converge to a point belonging  to $A\cap B\neq \varnothing$. 
\end{theorem}
\begin{proof}
Take any ${\bar x}\in A\cap B$. From  \eqref{eq:P_A(x^k)} we have  $y^{k+1}=P_B(x^k)$. Then,   applying   Lemma~\ref{le:projeccion} (ii)  with $C=B$, $v=x^k$,  and $z={\bar x}$,  we conclude  
\begin{equation}\label{pr:CaAlg1}
\|y^{k+1}-{\bar x}\|^2\leq \|x^k-{\bar x}\|^2-\|y^{k+1}-x^k\|^2.
\end{equation} 
Using \eqref{eq:P_B(y^k)} and applying  \eqref{eq:fep1} of Lemma \ref{pr:condi}  with $v=y^{k+1}$, $u=x^k$, $\gamma=\gamma_k$, $\theta=\theta_k$, $\lambda=\lambda_k$, $\varphi_k=\varphi_{ \gamma_k, \theta_k, \lambda_k}$, $w^+=x^{k+1}$, $z={\bar x}$,  and   $C=A$,  we obtain   
\begin{equation}\label{sequenceprop}
\|x^{k+1}-{\bar x}\|^2\leq \|y^{k+1}-{\bar x}\|^2+ \frac{2\gamma_k+2\lambda_k}{1-2\lambda_k}\|x^k-y^{k+1}\|^2- \frac{1-2\theta_k}{1-2\lambda_k} \|x^{k+1}-y^{k+1}\|^2.
\end{equation}
Therefore,  the combination of  \eqref{pr:CaAlg1} with \eqref{sequenceprop} yields  
\begin{equation} \label{eq:fipc}
\|x^{k+1}-{\bar x}\|^2\leq \|x^k-{\bar x}\|^2-  \frac{1-2\gamma_k-4\lambda_k}{1-2\lambda_k}\|x^k-y^{k+1}\|^2-  \frac{1-2\theta_k}{1-2\lambda_k} \|x^{k+1}-y^{k+1}\|^2.
\end{equation}
Hence,   \eqref{eq:fsa} and  \eqref{eq:fipc} imply that $(x^k)_{k\in\mathbb{N}}$ is F\'ejer convergent to $A\cap B\neq \varnothing$.  Thus, by Lemma~\ref{fejer}~(i), $(x^k)_{k\in\mathbb{N}}$ is bounded. Also, using Lemma~\ref{fejer}~(ii), we conclude  that $(\|x^k-{\bar x}\|)_{k\in\mathbb{N}}$ converges. Consequently,  from   \eqref{pr:CaAlg1} we obtain that $(y^k)_{k\in\mathbb{N}}$ is also bounded. 
 Since $A$ and $B$ are closed sets,  $(x^k)_{k\in\mathbb{N}}\subset  A$,  and $(y^k)_{k\in\mathbb{N}}\subset B$, all cluster cluster points of these  sequences  belong to the sets $A$ and $B$,  respectively. Now, using   \eqref{eq:fipc} together with  \eqref{eq:fsa},   we obtain  
 $$
 0< (1-\sigma)\|x^k-y^{k+1}\|^2\leq \|x^k-{\bar x}\|^2-  \|x^{k+1}-{\bar x}\|^2.
 $$
Since the  right hand side of the last inequality  converges   to zero,   $(\|y^{k+1}-x^k\|)_{k\in\mathbb{N}}$ also converges to zero. Thus all cluster points of $(x^k)_{k\in\mathbb{N}}$ are also clusters points of $(y^k)_{k\in\mathbb{N}}$, proving that there exists   ${\hat x}\in A\cap B$  a cluster point of $(x^k)_{k\in\mathbb{N}}$ and $(y^k)_{k\in\mathbb{N}}$. Therefore,  from Lemma \ref{fejer} (iii) we conclude that $(x^k)_{k\in\mathbb{N}}$  and  $(y^k)_{k\in\mathbb{N}}$ converge to a same  point in $A\cap B$.
\end{proof}
\subsection{The ACondG-1 method  for  two sets with  empty intersection} \label{sec:ACondG1e}
In the following  we assume that the sets $A$ and  $B$ have empty intersection, i.e., $A\cap B=  \varnothing$.  Since  $A$ is a compact set and the projection mapping $P_B$ is continuous we have
\begin{equation}\label{eq:bomef}
\zeta:=\sup\{ \|x-P_B(u)\|: ~x, u\in A\}<+\infty.
\end{equation}
Let us  also  {\it assume in this section  that  $(\lambda_k)_{k\in\mathbb{N}}$, $(\gamma_k)_{k\in\mathbb{N}}$ and $(\theta_k)_{k\in\mathbb{N}}$ are summable},  
\begin{equation} \label{eq:sssf}
 \sum_{k\in \mathbb{N}}\gamma_k< \infty, \qquad \quad  \sum_{k\in \mathbb{N}}\theta_k< \infty, \qquad \quad \sum_{k\in \mathbb{N}}\lambda_k< \infty.
\end{equation}
{\it  In addition, we consider      $0\leq \lambda_k<1/2$, for all $k=0, 1, \ldots$}.
\begin{theorem}
The sequences  $(x^k)_{k\in\mathbb{N}}$ and  $(y^k)_{k\in\mathbb{N}}$ converge,  respectively,  to  $x^*\in A$ and $y^*\in B$  satisfying  $\|x^*-y^*\|=dist(A,B)$. 
\end{theorem}
\begin{proof}
Since $(x^k)_{k\in\mathbb{N}}\subset A$ and   $y^{k+1}=P_B(x^{k})$, by  using \eqref{eq:bomef} we have  $\|y^{k+1}-x^{k}\|\leq \zeta$ and $\|x^{k+1}-y^{k+1}\|\leq \zeta$.  Thus,  applying inequality \eqref{eq:fep2} of  Lemma \ref{pr:condi}  with $v=y^{k+1}$, $u=x^{k}$, $\gamma=\gamma_{k}$, $\theta=\theta_{k}$, $\lambda=\lambda_{k}$, $\varphi_{k}=\varphi_{ \gamma_{k}, \theta_{k}, \lambda_{k}}$, $w^+=x^{k+1}$,  and $C=A$,  we obtain  
\begin{align*}
\|x^{k+1}-P_A(y^{k+1})\|^2&\leq  \frac{2\gamma_{k}+2\lambda_{k}}{1-2\lambda_{k}}\|y^{k+1}-x^k\|^2+\frac{2\theta_{k}}{1-2\lambda_{k}}  \|x^{k+1}-y^{k+1}\|^2\\
&\leq  \frac{2\gamma_{k}+2\theta_{k}+2\lambda_{k}}{1-2\lambda_{k}} \zeta^2.
\end{align*}
Since \eqref{eq:sssf} implies  that $\lim_{k\to +\infty}\gamma_k=0$, $\lim_{k\to +\infty}\theta_k=0$, and  $\lim_{k\to +\infty}\lambda_k=0$,  the last inequality   yields   
\begin{equation}\label{gotozerotoof}
\lim_{k\to +\infty} \|x^{k+1}-P_A(y^{k+1})\|=0.
\end{equation}
On the other hand, applying  Lemma~\ref{le:projeccion} (ii)  with $C=B$,  $v=x^k$,   and   $z= y^{k}\in B$, and  taking into account $\eqref{eq:P_A(x^k)}$,  we have
\begin{equation}\label{nointeykth1}
\|y^{k+1}-y^{k}\|^2\leq \|x^k-y^{k}\|^2- \|y^{k+1}-x^{k}\|^2.
\end{equation}
Applying  inequality \eqref{eq:fep1} of  Lemma \ref{pr:condi}  with $v=y^{k+1}$, $u=x^{k}$, $\gamma=\gamma_{k}$, $\theta=\theta_{k}$, $\lambda=\lambda_{k}$, $\varphi_{k}=\varphi_{ \gamma_{k}, \theta_{k}, \lambda_{k}}$, $w^+=x^{k+1}$, $z=x^{k}$ and $C=A$, and summing the obtained inequality with   \eqref{nointeykth1},   after some algebraic manipulations, we obtain 
\begin{multline*} 
\|x^{k+1}-x^{k}\|^2+\|y^{k+1}-y^{k}\|^2 \leq \|x^{k}-y^{k}\|^2- \|x^{k+1}-y^{k+1}\|^2 \\+\frac{2\gamma_{k}+2\lambda_{k}}{1-2\lambda_{k}}\|y^{k+1}-x^{k}\|^2+ \frac{2\theta_{k}-2\gamma_{k}}{1-2\lambda_{k}} \|x^{k+1}-y^{k+1}\|^2 .
\end{multline*}
Since   $\|y^{k+1}-x^{k}\|\leq \zeta$ and $\|x^{k+1}-y^{k+1}\|\leq \zeta$,  we conclude from the above  inequality   that 
$$
\|x^{k+1}-x^{k}\|^2+\|y^{k+1}-y^{k}\|^2 \leq \|x^{k}-y^{k}\|^2- \|x^{k+1}-y^{k+1}\|^2+\frac{2\gamma_{k}+2\theta_{k}+2\lambda_{k}}{1-2\lambda_{k}} \zeta^2.
$$
Thus, using \eqref{eq:sssf} and  that $\lim_{k\to +\infty}\lambda_k=0$,    the last inequality  yields 
$$
\sum_{k\in\mathbb{N}}\left(\|x^{k+1}-x^{k}\|^2+\|y^{k+1}-y^{k}\|^2\right)\leq \|x^{0}-y^0\|^2+ \zeta^2 \sum_{k\in\mathbb{N}} \frac{2\gamma_{k}+2\theta_{k}+2\lambda_{k}}{1-2\lambda_{k}} <+ \infty, 
$$
which implies  that   both sequences $(\|x^{k+1}-x^{k}\|^2)_{k\in\mathbb{N}}$ and $(\|y^{k+1}-y^{k}\|^2)_{k\in\mathbb{N}}$ converge  to zero. Hence, considering   \eqref{eq:P_A(x^k)} and \eqref{gotozerotoof},  we can apply Lemma~\ref{le:ppseq} with  $C=A$ and $D=B$, $v^k=x^k$ and $w^k=y^k$, for all $k=0, 1, \ldots$, to conclude that   each cluster point ${\bar x}$  of $(x^k)_{k\in\mathbb{N}}$ is a fixed point of $P_AP_B$, i.e.,  ${\bar x}=P_AP_B({\bar x})$,   $\lim_{k\to \infty}\|x^k-P_B(x^k)\|=dist(A,B)$ and  $\lim_{k\to \infty}(x^k-P_B(x^k))=P_{A-B}(0)$. 

Now,  we are going to prove that the whole sequence $(x^k)_{k\in\mathbb{N}}$ converges. For that, consider the set $E=\{x\in A: \|x-P_{B}(x)\|=dist(A,B)\}$.  We already proved that all clusters point of $(x^k)_{k\in\mathbb{N}}$ belong to $E$. Applying Corollary~\ref{cr:condi} with $v=y^{k+1}$, $u=x^{k}$, $w^+=x^{k+1}$, ${\bar z}={\bar x}\in E$,  $\gamma=\gamma_{k}$, $\theta=\theta_{k}$, $\lambda=\lambda_{k}$, $\varphi_{k}=\varphi_{ \gamma_{k}, \theta_{k}, \lambda_{k}}$,  $C=A$,  and $D=B$, we obtain 
 \begin{multline*} 
\|x^{k+1}-{\bar x}\|^2\leq \|x^k-{\bar x}\|^2 + 2\langle x^k-y^{k+1}, P_B({\bar x})-y^{k+1}\rangle\\+\frac{2\gamma_k+2\lambda_k}{1-2\lambda_k}\|y^{k+1}-x^k\|^2- \frac{2\lambda_k-2\theta_k}{1-2\lambda_k} \|x^{k+1}-y^{k+1}\|^2. 
\end{multline*}
Since  $y^{k+1}=P_B(x^{k})$, applying Lemma~\ref{le:projeccion} (i) with $v=x^{k}$ and $z=P_B({\bar x})$ we have  
$$
\|x^{k+1}-{\bar x}\|^2 \leq\|x^k-{\bar x}\|^2  +\frac{2\gamma_k+2\lambda_k}{1-2\lambda_k}\|y^{k+1}-x^k\|^2- \frac{2\lambda_k-2\theta_k}{1-2\lambda_k} \|x^{k+1}-y^{k+1}\|^2. 
$$
Therefore,  due to   $\|y^{k+1}-x^k\| \leq\zeta$,  $\|x^{k+1}-y^{k+1}\|\leq \zeta$ and $0\leq \lambda_k<1/2$, it follows that 
\begin{equation} \label{eq:qlaf}
\|x^{k+1}-{\bar x}\|^2 \leq\|x^k-{\bar x}\|^2   + \zeta^2 \frac{2\gamma_k+2\theta_k+2\lambda_k}{1-2\lambda_k}. 
\end{equation}
By using \eqref{eq:sssf} and taking into account that  $\lim_{k\to +\infty}\lambda_k=0$,  we obtain 
$$
\sum_{k\in\mathbb{N}} \zeta^2 \frac{2\gamma_k+2\theta_k+2\lambda_k}{1-2\lambda_k}< \infty, 
$$
which combined with \eqref{eq:qlaf}  implies that   $(x^k)_{k\in\mathbb{N}}$ is  quasi-F\'ejer convergent  to the set $E$. Since  the sequence $(x^k)_{k\in\mathbb{N}}$  has a cluster point  belonging  to $E$, it follows  that the whole sequence converge to a point $x^*\in E$. Hence, it follows from \eqref{eq:P_A(x^k)} that  $\lim_{k \to +\infty}y^{k+1}=P_B(x^*)$. Therefore, setting  $y^*=P_B(x^*)$ and  due to $x^*=P_AP_B(x^*)$  we also  have $y^*=P_BP_A(y^*)$. By using  \cite[Theorem 2]{CheneyGoldstein1959} we obtain  that  $\|y^*-P_A(y^*))\|=dist(A,B)$, which  concludes the proof. 
\end{proof}
\section{The ACondG method with  inexact  projections  onto two sets} \label{Sec:ACondG-2}
In this section,  we present our  second  version of inexact  alternating  projection method  to solve Problem~\eqref{def:InexactProjProb}, by using the CondG method for compute feasible inexact projections  with respect to  both sets in consideration. This method  will be  named as {\it alternating conditional gradient-2 (ACondG-2)} method. Let us  we assume that  $A$ and  $B$ are   convex and compact sets.  The  ACondG-2 method,  is formally defined as follows:

\begin{algorithm} 
\begin{description}
\item[ Step 0.] Let   $(\lambda_k)_{k\in\mathbb{N}}$, $(\gamma_k)_{k\in\mathbb{N}}$ and $(\theta_k)_{k\in\mathbb{N}}$  be  sequences of nonnegative real numbers and the associated function $\varphi_k:=\varphi_{ \gamma_k, \theta_k, \lambda_k}$, as defined in \eqref{eq:pfphi}.   Let $x_0\in A$, $y_0\in B$.  If $x^0\in B$ or $y^0\in A$, then {\bf stop}. Otherwise, initialize $k\leftarrow 0$.
\item[ Step 1.] Using   Agorithm \ref{Alg:CondG},  compute  $\mbox{CondG$_{B}$}\,(\varphi_k, y^k,  x^k)$ and set the next iterate  $y^{k+1}$ as  
\begin{equation} \label{eq2:P_B}
 y^{k+1}:=\mbox{CondG$_{B}$}\,(\varphi_k, y^k,  x^k).  
\end{equation} 
If   $y_{k+1}\in A$, then {\bf stop}. 
\item[ Step 2.]   Using   Agorithm \ref{Alg:CondG},  compute  $\mbox{CondG$_{A}$}\,(\varphi_k, x^k,y^{k+1})$ and and set the next iterate  $x^{k+1}$  as  
$$
 x^{k+1}:=\mbox{CondG$_{A}$}\,(\varphi_k, x^k,y^{k+1}). 
$$
If $x^{k+1}\in B$, then {\bf stop}.
\item[ Step 3.]  Set $k\gets k+1$, and go to \textbf{Step~1}.
\end{description}
\caption{ACondG method with  inexact  projection onto  two sets (ACondG-2)}
\label{Alg:ACondG-2}
\end{algorithm}

As for Algorithm~\ref{Alg:ACondG-1}, if Algorithm~\ref{Alg:ACondG-2} stops, it means that a point belonging to $A\cap B$ has been found.  Therefore, hereafter  we assume  that  $(x^k)_{k\in\mathbb{N}}$ and $(y^k)_{k\in\mathbb{N}}$  generated by Algorithm~\ref{Alg:ACondG-2} are  infinity sequences.   We will proceed with the convergence analysis  of ACondG-2 method  by considering the cases where  $A\cap B\neq \varnothing $ and $A\cap B= \varnothing$.
\subsection{  The ACondG-2 method for two sets with nonempty intersection} 

Let us  assume that  $A$ and  $B$ have nonempty intersection, that is, $A\cap B\neq  \varnothing$. Additionally,  suppose that  the forcing sequences   $(\lambda_k)_{k\in\mathbb{N}}$, $(\gamma_k)_{k\in\mathbb{N}}$ and $(\theta_k)_{k\in\mathbb{N}}$ satisfy
\begin{equation} \label{eq:fsatt}
\theta_k\leq {\bar \theta}< 1/4, \qquad  2\gamma_k+4\lambda_k< {\sigma}<1, \qquad 2\gamma_k+2\theta_k+2\lambda_k<\rho<1,   
\end{equation}
for all $k=0, 1, \ldots$, where $ {\bar \theta}$, $ {\sigma}$ and $\rho$ are positive real  constants. 
\begin{theorem}
The sequences   $(x^k)_{k\in\mathbb{N}}$ and  $(y^k)_{k\in\mathbb{N}}$ converge to a point belonging  to  set $A\cap B\neq \varnothing$. 
\end{theorem}
\begin{proof} 
Let $x^*\in A\cap B$ and $k\in\mathbb{N}$. Applying \eqref{eq:fep1} of Lemma \ref{pr:condi} with $v=x^k$, $u=y^k$, $\gamma=\gamma_k$, $\theta=\theta_k$, $\lambda=\lambda_k$, $\varphi_k=\varphi_{ \gamma_k, \theta_k, \lambda_k}$, $w^+=y^{k+1}$, $z=x^*$,   and $C=B$,  we have
\begin{equation}\label{ykclose}
\|y^{k+1}-x^*\|^2\leq \|x^k-x^*\|^2+ \frac{2\gamma_k+2\lambda_k}{1-2\lambda_k}\|x^{k}-y^k\|^2- \frac{1-2\theta_k}{1-2\lambda_k} \|y^{k+1}-x^{k}\|^2.
\end{equation}
On the other hand,  applying  \eqref{eq:fep1} of Lemma \ref{pr:condi} with $v=y^{k+1}$, $u=x^k$, $\gamma=\gamma_k$, $\theta=\theta_k$, $\lambda=\lambda_k$, $\varphi_k=\varphi_{ \gamma_k, \theta_k, \lambda_k}$, $w^+=x^{k+1}$, $z=x^*$,  and $C=A$,  we obtain 
\begin{equation}\label{xk+1}
\|x^{k+1}-x^*\|^2\leq \|y^{k+1}-x^*\|^2 +\frac{2\gamma_k+2\lambda_k}{1-2\lambda_k}\|y^{k+1}-x^k\|^2- \frac{1-2\theta_k}{1-2\lambda_k} \|x^{k+1}-y^{k+1}\|^2.
\end{equation}
Combining inequalities   \eqref{ykclose} and  \eqref{xk+1} we conclude that 
\begin{multline} \label{eq:cten}
\|x^{k+1}-x^*\|^2\leq  \|x^k-x^*\|^2-\frac{1-2\theta_k}{1-2\lambda_k}\|y^{k+1}-x^{k+1}\|^2  \\ + \frac{2\gamma_k+2\lambda_k}{1-2\lambda_k}\|y^k-x^{k}\|^2 -\frac{1-2\gamma_k-2\theta_k-2\lambda_k}{1-2\lambda_k}\|x^{k}-y^{k+1}\|^2.
\end{multline}
Thus, taking into account \eqref{eq:fsatt}, the  inequality \eqref{eq:cten} implies 
\begin{equation}\label{eq:fcy}
\|x^{k+1}-x^*\|^2+ \frac{1}{2}\|x^{k+1}-y^{k+1}\|^2 \leq \|x^k-x^*\|^2+ \frac{1}{2}\|x^k-y^{k}\|^2.
\end{equation}
In particular, \eqref{eq:fcy} implies that the sequence $(\|x^{k}-x^*\|^2+ \frac{1}{2}\|x^{k}-y^{k}\|^2)_{k\in\mathbb{N}}$ is non-increasing. Hence, it converges and, moreover,    $(x^k)_{k\in\mathbb{N}}$ is bounded. We also have from    \eqref{eq:fsatt} and  \eqref{eq:cten} that 
\begin{equation}\label{sumable}
 \|x^k-y^{k+1}\|^2\leq \frac{1}{\rho} \Big{[}  \Big{(}\|x^k-x^*\|^2+ \frac{1}{2}\|x^k-y^{k}\|^2  \Big{)} -    \Big{(}\|x^{k+1}-x^*\|^2+ \frac{1}{2}\|x^{k+1}-y^{k+1}\|^2 t \Big{)} \Big{]}. 
\end{equation}
In its turn,  applying  \eqref{eq:fep1} of  Lemma \ref{pr:condi} with $v=y^{k+1}$, $u=x^k$, $\gamma=\gamma_k$, $\theta=\theta_k$, $\lambda=\lambda_k$, $\varphi_k=\varphi_{ \gamma_k, \theta_k, \lambda_k}$, $w^+=x^{k+1}$, $z=x^k$ and $C=A$,  we have 
$$
\|x^{k+1}-x^k\|^2\leq \|y^{k+1}-x^k\|^2 +\frac{2\gamma_k+2\lambda_k}{1-2\lambda_k}\|y^{k+1}-x^k\|^2- \frac{1-2\theta_k}{1-2\lambda_k} \|x^{k+1}-y^{k+1}\|^2.
$$
Thus,  by using the two first inequalities  in  \eqref{eq:fsatt},    we have  $ \|x^{k+1}-x^k\|\leq 2\|y^{k+1}-x^k\|$, which  after apply     triangular  inequality,   yields  $\|x^{k+1}-y^{k+1}\| \leq  3 \|x^k-y^{k+1}\|$.
Hence, since  \eqref{sumable}  implies  $\sum_{k\in \mathbb{N}}\|x^k-y^{k+1}\|^2<+\infty$, we obtain  $\sum_{k\in\mathbb{N}}\|x^{k}-y^k\|^2<+\infty$.   In particular, this  inequality implies that  $(\|x^k-y^{k}\|)_{k\in\mathbb{N}}$ converges to zero.  Hence, $(x^k)_{k\in\mathbb{N}}$  and   $(y^k)_{k\in\mathbb{N}}$ has the same cluster points. Taking into account that  $(x^k)_{k\in\mathbb{N}}\subset A$  and   $(y^k)_{k\in\mathbb{N}}\subset B$,  we conclude that all cluster points of $(x^k)_{k\in\mathbb{N}}$  and $(y^k)_{k\in\mathbb{N}}$  are in $A\cap B\neq \varnothing$. Finally,  the combination of  \eqref{xk+1}   with \eqref{eq:fsatt} implies 
$$
\|x^{k+1}-x^*\|^2\leq  \|x^k-x^*\|^2  + 2\|x^{k}-y^k\|^2, 
$$
and  considering that  $\sum_{k\in\mathbb{N}}\|x^{k}-y^k\|^2<+\infty$, we conclude that $(x^k)_{k\in\mathbb{N}}$ is  quasi-F\'ejer convergence to $A\cap B\neq \varnothing$. Since all clusters point of  $(x^k)_{k\in\mathbb{N}}$  belongs to $A\cap B$, Lemma \ref{fejer}{\it (iii)} implies  that it converges   to a point in   $A\cap B$.  Therefore,  due to the  cluster points of $(x^k)_{k\in\mathbb{N}}$  and   $(y^k)_{k\in\mathbb{N}}$ are the same, the results follows and the proof is concluded. 
\end{proof} 
\subsection{The ACondG-2 method for  two sets with  empty intersection}
Now,  we assume that the sets $A$ and  $B$ have empty intersection, that is, $A\cap B=  \varnothing$.  Since  $A$ and $B$ are bounded we have
$$
\omega:=\sup\{ \|x-y\|: ~x\in A,~y\in B\}<+\infty.
$$
As in section~\ref{sec:ACondG1e},  we also assume that $(\lambda_k)_{k\in\mathbb{N}}$, $(\gamma_k)_{k\in\mathbb{N}}$ and $(\theta_k)_{k\in\mathbb{N}}$ are summable, i.e.,  satisfy  \eqref{eq:sssf}, and   that   $0\leq \lambda_k<1/2$, for all $k=0, 1, \ldots$.
\begin{theorem}
The sequences  $(x^k)_{k\in\mathbb{N}}$ and  $(y^k)_{k\in\mathbb{N}}$ converge respectively to $x^*\in A$ and $y^*\in B$ which satisfy $\|x^*-y^*\|=dist(A,B)$. 
\end{theorem}
\begin{proof}
Let  $x\in A$ and $y\in B$. Applying \eqref{eq:fep2} of Lemma \ref{pr:condi} with $v=x^k$, $u=y^k$, $w^+=y^{k+1}$, $\gamma=\gamma_k$, $\theta=\theta_k$, $\lambda=\lambda_k$, $\varphi_k=\varphi_{ \gamma_k, \theta_k, \lambda_k}$,   and $C=B$, we have
$$
\|y^{k+1}-P_B(x^k)\|^2\leq  \frac{2\gamma_k+2\lambda_k}{1-2\lambda_k}\|x^{k}-y^k\|^2+\frac{2\theta_k}{1-2\lambda_k} \|y^{k+1}-x^{k}\|^2.
$$
Considering that  $(x^k)_{k\in\mathbb{N}}\subset A$ and  $(y^k)_{k\in\mathbb{N}}\subset B$, we have  $\|x^{k}-y^k\|\leq \omega$ and $\|y^{k+1}-x^{k}\|\leq \omega$. Moreover,  \eqref{eq:sssf} implies that   $\lim_{k\to +\infty}\gamma_k=0$, $\lim_{k\to +\infty}\theta_k=0$, and  $\lim_{k\to +\infty}\lambda_k=0$. Thus,   from the above inequality, we obtain  
\begin{equation}\label{gotozero}
\lim_{k\to +\infty} \|y^{k+1}-P_B(x^k)\|^2=0.
\end{equation}
Applying again \eqref{eq:fep2} of Lemma \ref{pr:condi}  with $v=y^{k+1}$, $u=x^k$, $w^+=x^{k+1}$, $\gamma=\gamma_k$, $\theta=\theta_k$, $\lambda=\lambda_k$, $\varphi_k=\varphi_{ \gamma_k, \theta_k, \lambda_k}$ and $C=A$,  we can also    conclude  that  
\begin{equation}\label{gotozerotoo}
\lim_{k\to +\infty} \|x^{k+1}-P_A(y^{k+1})\|^2=0.
\end{equation}
On the other hand,  applying \eqref{eq:fep1} of Lemma \ref{pr:condi} with $v=x^k$, $u=y^k$, $w^+=y^{k+1}$ $z=y^k$, $\gamma=\gamma_k$, $\theta=\theta_k$, $\lambda=\lambda_k$, $\varphi_k=\varphi_{ \gamma_k, \theta_k, \lambda_k}$,   and $C=B$,   it follows that 
$$
\|y^{k+1}-y^k\|^2\leq \|x^k-y^k\|^2+ \frac{2\gamma_k+2\lambda_k}{1-2\lambda_k}\|x^{k}-y^k\|^2- \frac{1-2\theta_k}{1-2\lambda_k} \|y^{k+1}-x^{k}\|^2.
$$
Now, applying  \eqref{eq:fep1} of Lemma \ref{pr:condi}  with $v=y^{k+1}$, $u=x^k$, $w^+=x^{k+1}$, $z=x^k$, $\gamma=\gamma_k$, $\theta=\theta_k$, $\lambda=\lambda_k$, $\varphi_k=\varphi_{ \gamma_k, \theta_k, \lambda_k}$,  and $C=A$, we obtain 
$$
\|x^{k+1}-x^k\|^2\leq \|y^{k+1}-x^k\|^2 +\frac{2\gamma_k+2\lambda_k}{1-2\lambda_k}\|y^{k+1}-x^k\|^2- \frac{1-2\theta_k}{1-2\lambda_k} \|x^{k+1}-y^{k+1}\|^2.
$$
Summing the above   two previous  inequalities we conclude 
\begin{multline*}
\|x^{k+1}-x^k\|^2+\|y^{k+1}-y^k\|^2 \leq  \frac{1+2\gamma_k}{1-2\lambda_k}\|x^{k}-y^k\|^2\\ -\frac{1-2\theta_k}{1-2\lambda_k}\|x^{k+1}-y^{k+1}\|^2+ \frac{2\theta_k+2\gamma_k}{1-2\lambda_k} \|y^{k+1}-x^{k}\|^2.
\end{multline*}
Thus, considering that   $\|x^{k+1}-y^{k+1}\|\leq \omega$,   $\|x^{k}-y^k\|\leq \omega$,  and   $\|y^{k+1}-x^{k}\|\leq \omega$ and  after some algebraic manipulations, we have 
$$
\|x^{k+1}-x^k\|^2+\|y^{k+1}-y^k\|^2 \leq  \|x^{k}-y^k\|^2 -\|x^{k+1}-y^{k+1}\|^2+\frac{4\gamma_k+4\theta_k+2\lambda_k}{1-2\lambda_k}\omega^2.
$$
Consequently, using \eqref{eq:sssf} and  that $\lim_{k\to +\infty}\lambda_k=0$,  we obtain
\begin{multline}\label{sumable2}
\sum_{k\in\mathbb{N}}\left(\|x^{k+1}-x^k\|^2+\|y^{k+1}-y^k\|^2\right)\leq \|x^{0}-y^0\|^2+ \omega^2 \sum_{k\in\mathbb{N}} \frac{4\gamma_k+4\theta+2\lambda_k}{1-2\lambda_k} <+ \infty,  
\end{multline}
which implies  that  $(\|x^{k+1}-x^{k}\|^2)_{k\in\mathbb{N}}$ and $(\|y^{k+1}-y^{k}\|^2)_{k\in\mathbb{N}}$ converge  to zero. Hence, considering  \eqref{gotozero} and  \eqref{gotozerotoo},  we can apply Lemma~\ref{le:ppseq} with  $C=A$ and $D=B$, $v^k=x^k$ and $w^k=y^k$, for all $k=0, 1, \ldots$, to conclude that   each cluster point ${\bar x}$  of $(x^k)_{k\in\mathbb{N}}$ is a fixed point of $P_AP_B$, i.e.,  ${\bar x}=P_AP_B({\bar x})$,   $\lim_{k\to \infty}\|x^k-P_B(x^k)\|=dist(A,B)$ and  $\lim_{k\to \infty}(x^k-P_B(x^k))=P_{A-B}(0)$.

 Now,  we are going to prove that the whole sequence $(x^k)_{k\in\mathbb{N}}$ converges. For that, consider the set $E=\{x\in A: \|x-P_{B}(x)\|=dist(A,B)\}$.  We already proved that all clusters point of $(x^k)_{k\in\mathbb{N}}$ belong to $E$. Take ${\bar x}\in E$.   Applying Corollary~\ref{cr:condi}  with $v=y^{k+1}$, $u=x^k$, $w^+=x^{k+1}$, ${\bar z}={\bar x}$, $\gamma=\gamma_k$, $\theta=\theta_k$, $\lambda=\lambda_k$, $\varphi_k=\varphi_{ \gamma_k, \theta_k, \lambda_k}$, $C=A$ and $D=B$, we obtain 
 \begin{multline*}  \label{eq:stsi}
\|x^{k+1}-{\bar x}\|^2 \leq\|x^k-{\bar x}\|^2 + 2\langle x^k-y^{k+1}, P_B({\bar x})-y^{k+1}\rangle \\+\frac{2\gamma_k +2\theta_k}{1-2\lambda_k}\|y^{k+1}-x^k\|^2- \frac{2\lambda_k-2\theta_k}{1-2\lambda_k} \|x^{k+1}-y^{k+1}\|^2. 
\end{multline*}
Now,   by using \eqref{eq2:P_B} we have   $y^{k+1}=\mbox{CondG$_{B}$}\,(\varphi_k, y^k,  x^k)$. Hence, it follows from \eqref{inexactset} that $\langle x^k-y^{k+1}, P_B(\bar{x})-y^{k+1}\rangle\leq \varphi_{ \gamma_k, \theta_k, \lambda_k}(y^k, x^k, y^{k+1})$. Then, from \eqref{eq:pfphi} we have 
$$
\langle x^k-y^{k+1}, P_B(\bar{x})-y^{k+1}\rangle\leq \gamma_k \|x^k-y^k\|^2 +\theta_k  \|y^{k+1}-x^k\|^2 +   \lambda_k \|y^{k+1}-y^k\|^2. 
$$
Therefore,  due to   $\|x^{k}-y^{k}\|\leq \omega$,  $\|y^{k+1}-x^{k}\|\leq \omega$, $ \|x^{k+1}-y^{k+1}\|\leq \omega$ and $0\leq \lambda_k<1/2$, it follows from the last inequality and \eqref{eq:stsi} that 
\begin{equation} \label{eq:qla}
\|x^{k+1}-{\bar x}\|^2 \leq   \|x^k-{\bar x}\|^2 + \omega^2\left(2\gamma_k+2\theta_k +\frac{2\gamma_k+4\theta_k}{1-2\lambda_k} \right)+\|y^{k+1}-y^k\|^2. 
\end{equation}
By using \eqref{eq:sssf}, \eqref{sumable2} and that  $\lim_{k\to +\infty}\lambda_k=0$,  we obtain 
$$
\sum_{k\in\mathbb{N}}\left[\omega^2\left(2\gamma_k+2\theta_k +\frac{2\gamma_k+4\theta_k}{1-2\lambda_k} \right)+\|y^{k+1}-y^k\|^2\right]< \infty, 
$$
which combined with \eqref{eq:qla}  implies that   $(x^k)_{k\in\mathbb{N}}$ is  quasi-F\'ejer convergent  to  $E$. Since  the sequence $(x^k)_{k\in\mathbb{N}}$  has a cluster point  belonging  to $E$, it follows  that the whole sequence converges to a point $x^*\in E$.  Finally, we also know that $(x^k-y^k)_{k\in\mathbb{N}}$ converges. Considering that  $y^k=x^k+(y^k-x^k)\in B$, for all $k=0, 1, \ldots$, we conclude that $(y^k)_{k\in\mathbb{N}}$ also converges to a point $y^*\in B$.  Hence, it follows from \eqref{gotozero} that  $\lim_{k \to +\infty}y^{k}=P_B(x^*)$. Therefore,  $y^*=P_B(x^*)$ and  due to $x^*=P_AP_B(x^*)$  we also  have $y^*=P_BP_A(y^*)$  and,  by using  \cite[Theorem 2]{CheneyGoldstein1959},  we obtain  $\|y^*-P_A(y^*))\|=dist(A,B)$, which  concludes the proof. 
\end{proof}

\section{Numerical examples} \label{Sec:NumExp}

The purpose of this section is  illustrate the practical behavior and demonstrate the potential advantages of the ACondG-1 and ACondG-2 algorithms over their exact counterparts. The exact schemes correspond to Algorithm~\ref{Alg:ACondG-1} and Algorithm~\ref{Alg:ACondG-2} where the projections are calculated exactly. More specifically, $y^{k+1}:=P_B(x^k)$ in Step~1 of Algorithm~\ref{Alg:ACondG-2} and $x^{k+1}:=P_A(y^{k+1})$ in Step~2 of Algorithms~\ref{Alg:ACondG-1} and \ref{Alg:ACondG-2}. The exact projection of a vector $v$ onto a set $C$ is computed by solving the convex quadratic problem
\begin{equation}\label{exactproj}
\min_{z \in  C} \frac{1}{2}\|z-v\|^2.
\end{equation}
For future reference, we will call the exact schemes by ExactAlg\ref{Alg:ACondG-1} and ExactAlg\ref{Alg:ACondG-2} corresponding to ACondG-1 and ACondG-2, respectively.
All codes were implemented in Fortran~90 and are freely available at \url{https://orizon.ime.ufg.br/}.

 In our implementations, sets $A$ and $B$ are described in the general form
$$A:=\left\{z\in\mathbb{R}^n \colon  h_A(z)=0,~  g_A(z)\leq 0\right\},  \quad \quad B:=\left\{z\in\mathbb{R}^n \colon  h_B(z)=0,~  g_B(z)\leq 0\right\},$$
where $h_A\colon\mathbb{R}^n\to\mathbb{R}^{m_A}$, $g_A\colon\mathbb{R}^n\to\mathbb{R}^{p_A}$, $h_B\colon\mathbb{R}^n\to\mathbb{R}^{m_B}$, and $g_B\colon\mathbb{R}^n\to\mathbb{R}^{p_B}$ are continuously differentiable functions. The feasibility violations at a given point $z\in\mathbb{R}^n$ with respect to sets $A$ and $B$ are measured, respectively, by
$$c_{A}(z):=\max\left\{\|h_A(z)\|_{\infty},\|V_A(z)\|_\infty \right\} \quad \mbox{and} \quad c_{B}(z):=\max\left\{\|h_B(z)\|_{\infty},\|V_B(z)\|_\infty \right\},$$
where
$$[V_{A}(z)]_i = \max\{0,[g_A(z)]_i\},   \quad \quad [V_B(z)]_j = \max\{0,[g_B(z)]_j\},$$
 $i=1,\ldots,p_A,$ and $ j=1,\ldots,p_B$.  The algorithms are successfully stopped at iteration $k$, declaring that a feasible point was found, if
$$c_B(x^{k+1})\leq \varepsilon_{feas} \quad \mbox{or} \quad c_A(y^{k+1})\leq \varepsilon_{feas},$$
where $\varepsilon_{feas}>0$ is an algorithmic parameter.
We also consider a stopping criterion related to lack of progress: the algorithms terminate if, for two consecutive iterations, it holds that
$$\|x^{k+1}-x^k\|_{\infty}\leq \varepsilon_{lack} \quad \mbox{and} \quad \|y^{k+1}-y^k\|_{\infty}\leq \varepsilon_{lack},$$
where $\varepsilon_{lack}>0$ is also an algorithmic parameter. Note that the latter criterion should be satisfied if the intersection between sets A and B is empty.

In the CondG scheme given by Algorithm~\ref{Alg:CondG}, for computing the optimal solution $z_{\ell}$ at Step~1, we used the software Algencan~\cite{algencan}, an augmented Lagrangian code for general nonlinear optimization programming. Algencan was also used to solve \eqref{exactproj} in the exact versions of the algorithms. 

Function $\varphi_{ \gamma, \theta, \lambda}$ related to the degree of inexactness of the projections was set to be equal to the right hand side of \eqref{eq:pfphi}. The forcing sequences $(\gamma_k)_{k\in\mathbb{N}}$, $(\theta_k)_{k\in\mathbb{N}}$, and $(\lambda_k)_{k\in\mathbb{N}}$ are defined in an adaptive manner. We first choose $\gamma_0$, $\theta_0$, and $\lambda_0$ satisfying either condition \eqref{eq:fsa} or \eqref{eq:fsatt} for ACondG-1 and ACondG-2 algorithms, respectively. For the subsequent iterations if, between two consecutive iterations, enough progress is observed in terms of feasibility with respect to the sets $A$ or $B$, the parameters are not updated. Otherwise, the parameters are decreased by a fixed factor. 
This means that when a lack of progress is verified, the forcing parameters are decreased requiring more accurate projections. Formally, for $k\geq 1$, we set
$$(\gamma_k,\theta_k,\lambda_k) := \left\{\hspace{-0.2cm}\begin{array}{rl}
                (\gamma_{k-1},\theta_{k-1},\lambda_{k-1}), & \hspace{-0.2cm} \mbox{if } c_B(x^k) \leq \tau c_B(x^{k-1}) \mbox{ or } c_A(y^k) \leq \tau c_A(y^{k-1}), \\
                \delta (\gamma_{k-1},\theta_{k-1},\lambda_{k-1}), & \hspace{-0.2cm} \mbox{otherwise,} \\
               \end{array}\right.$$
where $\tau,\delta \in (0,1)$ are algorithmic parameters. Since $(\gamma_k)_{k\in\mathbb{N}}$, $(\theta_k)_{k\in\mathbb{N}}$, and $(\lambda_k)_{k\in\mathbb{N}}$ are non-increasing sequences, either condition \eqref{eq:fsa} or \eqref{eq:fsatt}, according to the chosen method, holds for all $k\geq 0$. Moreover, observe that if $A\cap B=  \varnothing$, then there exists $k_0\in\N$ such that $(\gamma_k,\theta_k,\lambda_k) = \delta (\gamma_{k-1},\theta_{k-1},\lambda_{k-1})$ for all $k>k_0$. Thus, $(\gamma_k,\theta_k,\lambda_k) = \delta^{k-k_0} (\gamma_{k_0},\theta_{k_0},\lambda_{k_0})$ for all $k>k_0$. As consequence, $(\lambda_k)_{k\in\mathbb{N}}$, $(\gamma_k)_{k\in\mathbb{N}}$ and $(\theta_k)_{k\in\mathbb{N}}$ are summable, because $\delta \in (0,1)$.

In our tests we set $\varepsilon_{feas}=\varepsilon_{lack}=10^{-8}$, $\gamma_0=0.1-\varepsilon_{feas}$, $\theta_0=\lambda_0=0.2-\varepsilon_{feas}$, $\tau=0.9$, and $\delta=0.1$.
\subsection{ACondG-1 algorithm} \label{sec:numericalACondG-1}
In this subsection we consider the problem of finding a point in the intersection of a region delimited by an ellipse and a half-plane. Given $z_0^A\in\R^2$, $\theta\in[-\pi,\pi]$, $a, b>0$, and defining 
\begin{equation}\label{DR}
D(a,b):=\begin{bmatrix}
  1/a^2       & 0\\
0      &  1/b^2
\end{bmatrix} \quad \mbox{and} \quad R(\theta):=\begin{bmatrix}
  \cos \theta       & \sin \theta \\
-\sin \theta       &  \cos \theta
\end{bmatrix},
\end{equation}
we take
\begin{equation}\label{setA-1}
A=\left\{z\in\mathbb{R}^2 \colon \ds \left\langle D(a,b)R(\theta) (z-z_0^A), ~R(\theta) (z-z_0^A)\right\rangle -1  \leq 0\right\}.
\end{equation}
Let  $\beta\in\R$ a parameter and define 
\begin{equation}\label{setB-1}
B=\left\{z\in\mathbb{R}^2 \colon -[z]_1 + \beta  \leq 0\right\}.
\end{equation}
Clearly, there exists an explicit expression for the projection onto $B$.  It is worth mentioning that ellipses satisfy the assumptions  of Theorem~\ref{eq:rccgm2}; see \cite[Lemma 2]{GarberHazan2015}. In our tests, we set $z_0^A=[0,0]^T$, $\theta = -\pi/4$, $a=2$, and $b=1/5$ in \eqref{setA-1} for defining $A$, and considered different values of $\beta$ in \eqref{setB-1} for $B$. Parameter $\beta$ allows to control the existence of points in the intersection of sets $A$ and $B$. 

Table~\ref{tab:ACondG-1} shows the performance of ACondG-1 and ExactAlg\ref{Alg:ACondG-1} algorithms on eight instances of the considered problem, while Figure~\ref{fig:ACondG-1} illustrates some ``solutions''. The initial point $x^0$ of each instance was taken to be the center $z_0^A$. In the table, the first column contains the considered values of $\beta$ and ``$A\cap B$'' informs whether the intersection of sets $A$ and $B$ is empty or not. Observe that in the first four instances there are points at the intersection of $A$ and $B$, while in the last four the intersection is empty. For each algorithm, ``SC'' informs the satisfied stopping criterion where ``C'' denotes {\it convergence} meaning that the algorithm found a point in $A\cap B$ and ``L'' means that the algorithms stopped due to lack of progress, ``it'' is the number of iterations, and ``$\min\{c_B(x^*),c_A(y^*)\}$'' is the smallest feasibility violation with respect to sets $A$ and $B$ at the final iterates.

\begin{table}[h!]
 {\footnotesize
\centering
\begin{tabular}{cc|ccc|ccc|}  \hhline{~~|*6{-}|}
\multicolumn{2}{  c|  }{}     &        \multicolumn{3}{c|}{\cellcolor[gray]{.9} ACondG-1} &  \multicolumn{3}{c|}{\cellcolor[gray]{.9} ExactAlg\ref{Alg:ACondG-1}} \\    \hline \rowcolor[gray]{.9}
 \multicolumn{1}{ |c  }{$\beta$} &  $A\cap B$  &  SC & it & $\min\{c_B(x^*),c_A(y^*)\}$ & SI & it & $\min\{c_B(x^*),c_A(y^*)\}$\\   \hline
 \multicolumn{1}{ |c  }{1.30}    &                                        & C & 5   & 0.00D+00 & L & 46  & 1.47D-08 \\
 \multicolumn{1}{ |c  }{1.35}    &                                        & C & 20  & 0.00D+00 & L & 53  & 1.44D-08 \\
 \multicolumn{1}{ |c  }{1.40}    &                                        & C & 29  & 0.00D+00 & L & 78  & 2.11D-08 \\
 \multicolumn{1}{ |c  }{1.42}    &  \multirowcell{-4}{$\neq \varnothing$} & C & 120 & 0.00D+00 & L & 348 & 5.67D-08 \\ \hline
 \multicolumn{1}{ |c  }{1.43}    &                                        & L & 45  & 8.73D-03 & L & 110 & 8.73D-03 \\
 \multicolumn{1}{ |c  }{1.45}    &                                        & L & 24  & 2.87D-02 & L & 49  & 2.87D-02 \\
 \multicolumn{1}{ |c  }{1.50}    &                                        & L & 19  & 7.87D-02 & L & 28  & 7.87D-02 \\
 \multicolumn{1}{ |c  }{1.60}    &  \multirowcell{-4}{$= \varnothing$}    & L & 9   & 1.79D-01 & L & 19  & 1.79D-01 \\ \hline
\end{tabular}
\caption{Performance of ACondG-1 and ExactAlg\ref{Alg:ACondG-1} algorithms on eight instances of the problem of finding a point in $A\cap B$, where the sets are given as $\eqref{setA-1}$ and $\eqref{setB-1}$.}
\label{tab:ACondG-1}}
\end{table}

\begin{figure}[h!]
\begin{minipage}[b]{0.49\linewidth}
\begin{center}
(a) ACondG-1 ($\beta=1.30$)
\end{center}\vspace{-0.80cm}
\begin{figure}[H]
	\centering
		
		\includegraphics[scale=0.50]{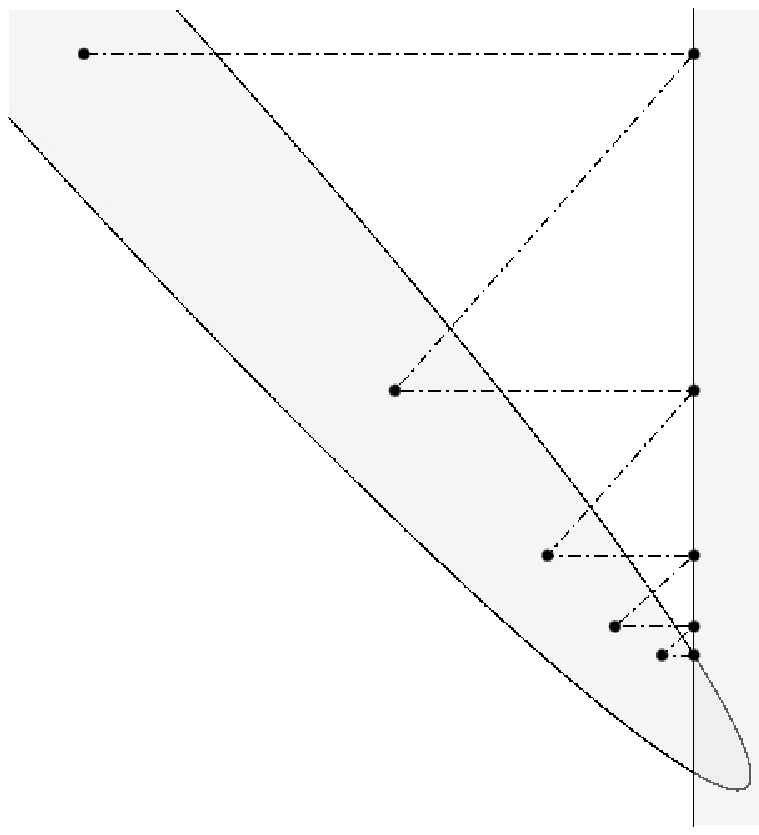}\\
\end{figure}

\end{minipage} \hfill
\begin{minipage}[b]{0.49\linewidth}
\begin{center}
(b) ExactAlg\ref{Alg:ACondG-1} ($\beta=1.30$)
\end{center}\vspace{-0.80cm}
\begin{figure}[H]
	\centering
		\includegraphics[scale=0.50]{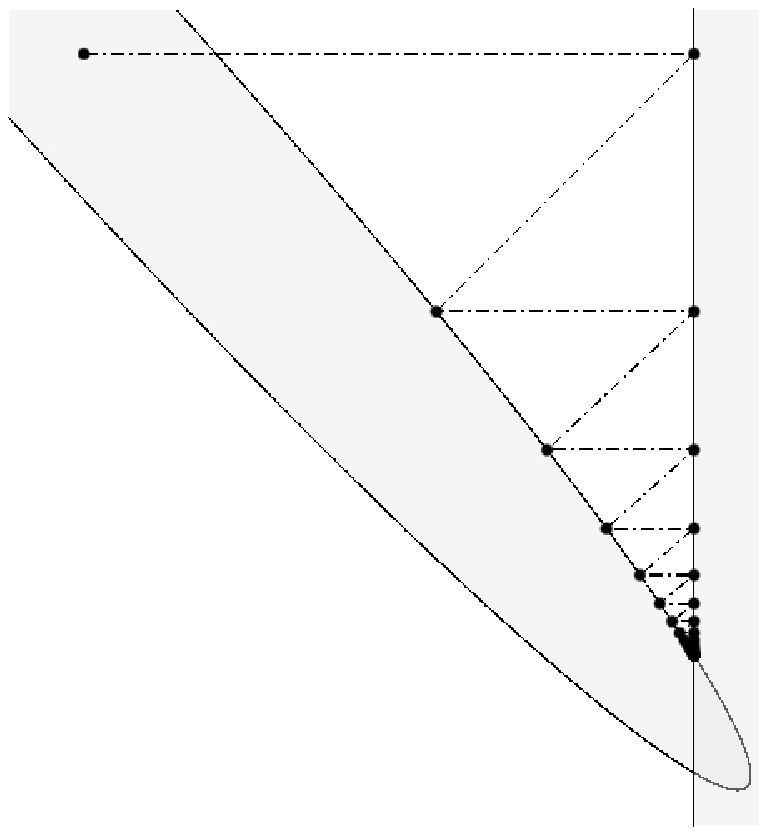}\\
\end{figure}
\end{minipage}\hfill

\vspace{0.3cm}

\begin{minipage}[b]{0.49\linewidth}
\begin{center}
(c) ACondG-1 ($\beta=1.50$)
\end{center}\vspace{-0.80cm}
\begin{figure}[H]
	\centering
		\includegraphics[scale=0.50]{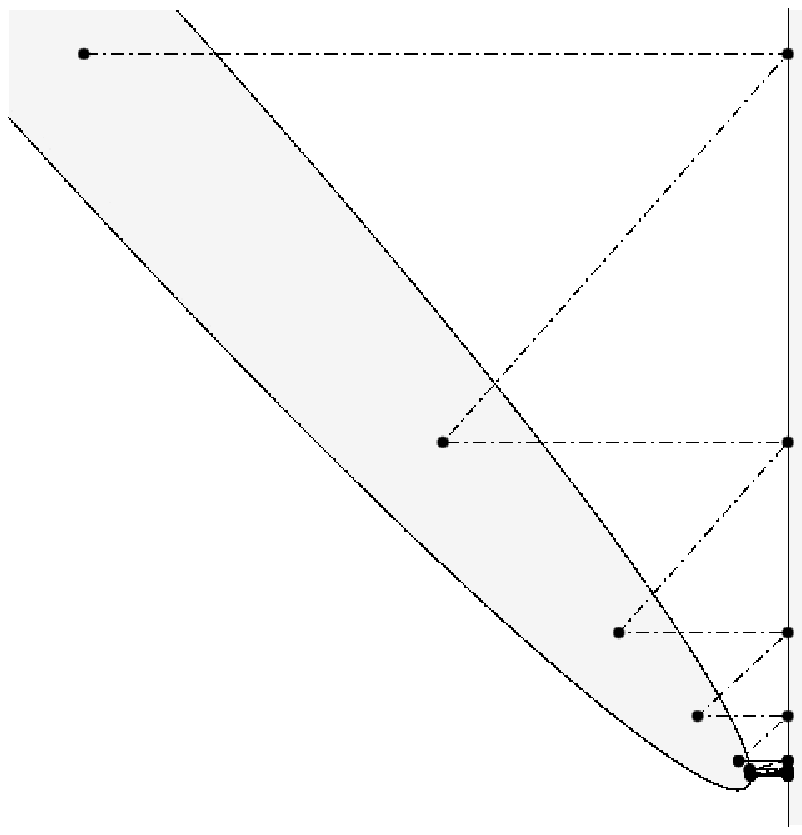}\\
\end{figure}

\end{minipage} \hfill
\begin{minipage}[b]{0.49\linewidth}
\begin{center}
(d) ExactAlg\ref{Alg:ACondG-1} ($\beta=1.50$)
\end{center}\vspace{-0.80cm}
\begin{figure}[H]
	\centering
		\includegraphics[scale=0.50]{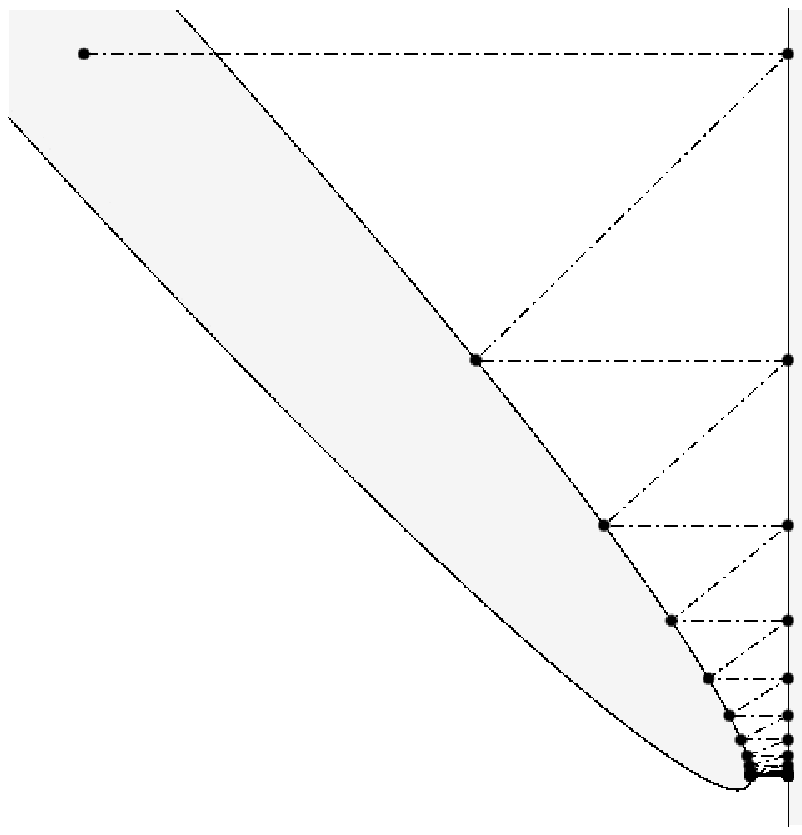}\\
\end{figure}
\end{minipage}\hfill
\caption{Behavior of ACondG-1 and ExactAlg\ref{Alg:ACondG-1} algorithms on the problem of finding a point in $A\cap B$, where the sets are given as $\eqref{setA-1}$ and $\eqref{setB-1}$ with $\beta=1.30$ and $\beta=1.50$.}
\label{fig:ACondG-1}
\end{figure}

In the first four instances, ACondG-1 algorithm converged to a solution. In fact, in these cases, we emphasize that a feasible point has been found, {\it not just} an infeasible point satisfying the prescribed feasibility tolerance. On the other hand, ExactAlg\ref{Alg:ACondG-1} failed to achieve the required feasibility tolerance $\varepsilon_{feas}=10^{-8}$ in all instances, stopping for lack of progress.
For ExactAlg\ref{Alg:ACondG-1}, we point out that the feasibility violation measure $\min\{c_B(x^{k+1}),c_A(y^{k+1})\}$ arrived $\mathcal{O}(10^{-8})$ after  39, 45, 69, and 331 iterations for $\beta = 1.30$, $1.35$, $1.40$, and $1.42$, respectively. This mean that ACondG-1 used $87.2\%$, $55.6\%$, $58.0\%$, and $63,4\%$ fewer iterations for finding a feasible point than ExactAlg\ref{Alg:ACondG-1} for achieve $\mathcal{O}(10^{-8})$ in the  feasibility violation measure. This seems surprising at first, because in most applications, an inexact algorithm has a low cost per iteration compared to its exact version, although the overall number of iterations tends to increase.
Figures~\ref{fig:ACondG-1}(a) and (b) show the behavior of ACondG-1 and ExactAlg\ref{Alg:ACondG-1} algorithms for $\beta=1.30$. As can be seen, the iterates of ExactAlg\ref{Alg:ACondG-1} are always at the boundary of set $A$, getting stuck at infeasible points {\it near} the intersection of the sets. In contrast, ACondG-1 goes into the interior of set $A$ approaching the intersection faster and reaching a feasible point.

In the four infeasible instances, as expected, both algorithms stopped for lack of progress reaching the same feasibility violation measure. However, as shown in Table~\ref{tab:ACondG-1}, ACondG-1 required, on average, $48.7\%$ fewer iterations than ExactAlg\ref{Alg:ACondG-1} to stop. Figures~\ref{fig:ACondG-1}(c) and (d) show the behavior of ACondG-1 and ExactAlg\ref{Alg:ACondG-1} algorithms for $\beta=1.50$. Observe that, as $k$ grows, the forcing parameters go to zero and the iterations of ACondG-1 become exact ones. This can be seen by noting that the iterates $x^k$, for large $k$, belong to the boundary of the ellipse.

\subsection{ACondG-2 algorithm}

Now we consider the problem of finding a point in the intersection of regions delimited by two ellipses. Let $A$ as in section~\ref{sec:numericalACondG-1}. For defining $B$, let $z_0^B\in\R^2$ and set
\begin{equation}\label{setB-2}
B=\left\{z\in\mathbb{R}^2 \colon \ds \left\langle D(c,d)R(\vartheta) (z-z_0^B), ~R(\vartheta) (z-z_0^B)\right\rangle -1  \leq 0\right\},
\end{equation}
where $D(\cdot,\cdot)$ and $R(\cdot)$ are given as \eqref{DR}, $\vartheta = \pi/3$, $c=2$, and $d=2/5$. 

In our tests, we defined $z_0^B=[\cdot,1/2]^T$ and considered different values for the first coordinate $[z_0^B]_1$. 
As for parameter $\beta$ in section~\ref{sec:numericalACondG-1}, $[z_0^B]_1$ can be used to determine whether or not there are points at the intersection of $A$ and $B$.
Table~\ref{tab:ACondG-2} shows the performance of ACondG-2 and ExactAlg\ref{Alg:ACondG-2} algorithms on eight instances of the problem corresponding to the values of $[z_0^B]_1$ given in the first column. The remaining columns are as defined for Table~\ref{tab:ACondG-1}. In each instance, the initial points were taken to be the centers of the ellipses, i.e., $x^0=z_0^A$ and $y^0=z_0^B$.
Figure~\ref{fig:ACondG-2} shows the behavior of the methods in the particular cases where $[z_0^B]_1=2.30$ and $[z_0^B]_1=2.50$.

\begin{table}[h!]
 {\footnotesize
\centering
\begin{tabular}{cc|ccc|ccc|}  \hhline{~~|*6{-}|}
\multicolumn{2}{  c|  }{}     &        \multicolumn{3}{c|}{\cellcolor[gray]{.9} ACondG-2} &  \multicolumn{3}{c|}{\cellcolor[gray]{.9} ExactAlg\ref{Alg:ACondG-2}} \\    \hline \rowcolor[gray]{.9}
 \multicolumn{1}{ |c  }{$[z_0^B]_1$} &  $A\cap B$  &  SC & it & $\min\{c_B(x^*),c_A(y^*)\}$ & SI & it & $\min\{c_B(x^*),c_A(y^*)\}$\\   \hline
 \multicolumn{1}{ |c  }{2.30 }   &                                        & C & 2   & 0.00D+00 & L & 40   & 2.71D-08 \\
 \multicolumn{1}{ |c  }{2.35 }   &                                        & C & 2   & 0.00D+00 & L & 127  & 4.60D-08 \\
 \multicolumn{1}{ |c  }{2.357}   &                                        & C & 8   & 0.00D+00 & L & 398  & 7.63D-08 \\
 \multicolumn{1}{ |c  }{2.358}   & \multirowcell{-4}{$\neq \varnothing$}  & C & 155 & 0.00D+00 & L & 699  & 1.06D-07 \\ \hline
 \multicolumn{1}{ |c  }{2.359}   &                                        & L & 724 & 1.50D-04 & L & 8378  & 7.31D-05 \\
 \multicolumn{1}{ |c  }{2.36 }   &                                        & L & 304 & 1.01D-03 & L & 1091  & 1.00D-03 \\ 
 \multicolumn{1}{ |c  }{2.40 }   &                                        & L & 23  & 4.01D-02 & L & 57  & 4.01D-02 \\
 \multicolumn{1}{ |c  }{2.50 }   &  \multirowcell{-4}{$= \varnothing$}    & L & 15  & 1.59D-01 & L & 25  & 1.59D-01 \\ \hline
\end{tabular}
\caption{Performance of ACondG-2 and ExactAlg\ref{Alg:ACondG-2} algorithms on eight instances of the problem of finding a point in $A\cap B$, where the sets are given as $\eqref{setA-1}$ and $\eqref{setB-2}$.}
\label{tab:ACondG-2}}
\end{table}

\begin{figure}[h!]
\begin{minipage}[b]{0.49\linewidth}
\begin{center}
(a) ACondG-2 ($[z_0^B]_1=2.30$)
\end{center}\vspace{-0.80cm}
\begin{figure}[H]
	\centering
		
		\includegraphics[scale=0.50]{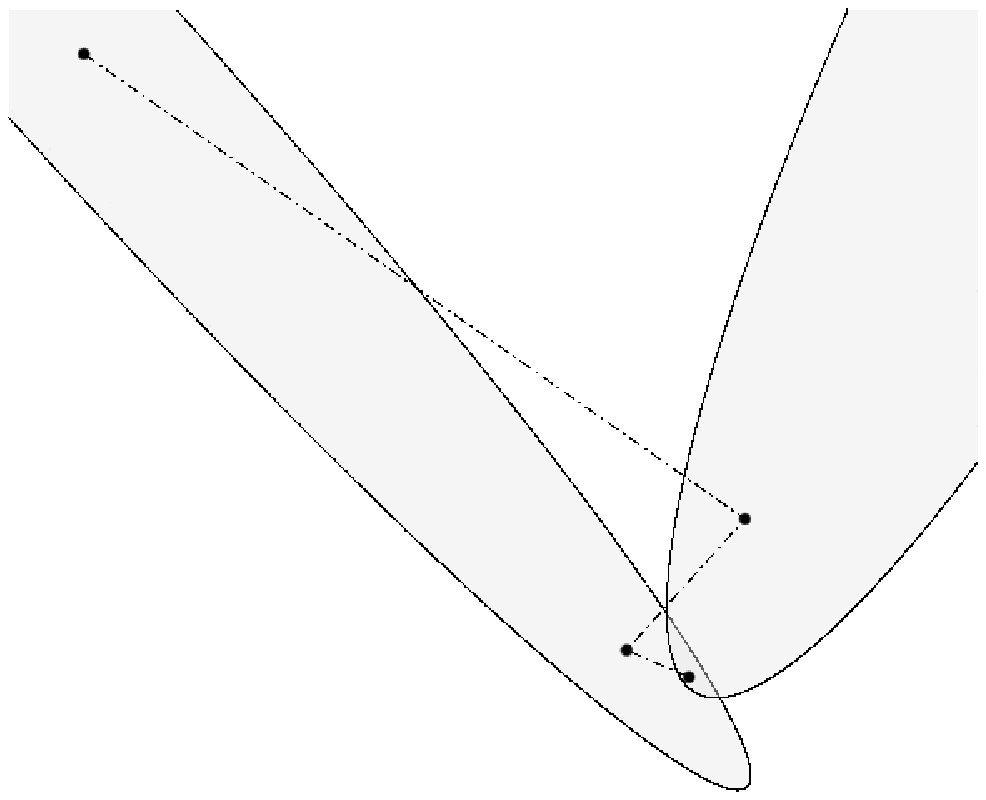}\\
\end{figure}

\end{minipage} \hfill
\begin{minipage}[b]{0.49\linewidth}
\begin{center}
(b) ExactAlg\ref{Alg:ACondG-2} ($[z_0^B]_1=2.30$)
\end{center}\vspace{-0.80cm}
\begin{figure}[H]
	\centering
		\includegraphics[scale=0.50]{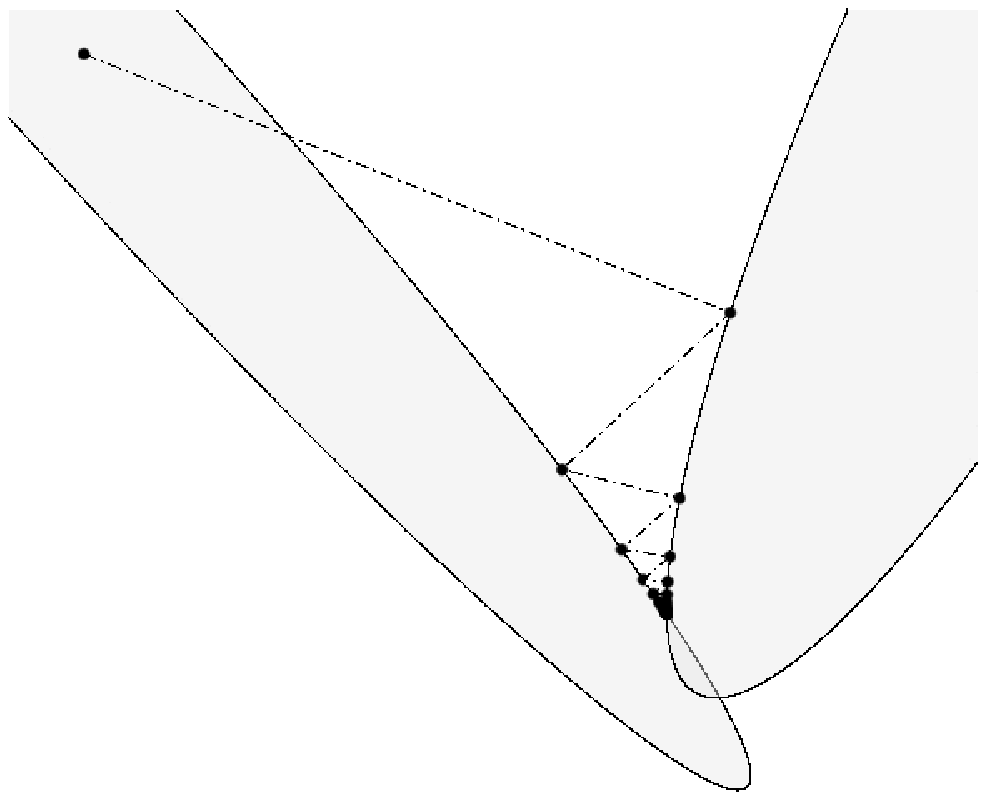}\\
\end{figure}
\end{minipage}\hfill

\vspace{0.3cm}

\begin{minipage}[b]{0.49\linewidth}
\begin{center}
(c) ACondG-2 ($[z_0^B]_1=2.50$)
\end{center}\vspace{-0.80cm}
\begin{figure}[H]
	\centering
		\includegraphics[scale=0.50]{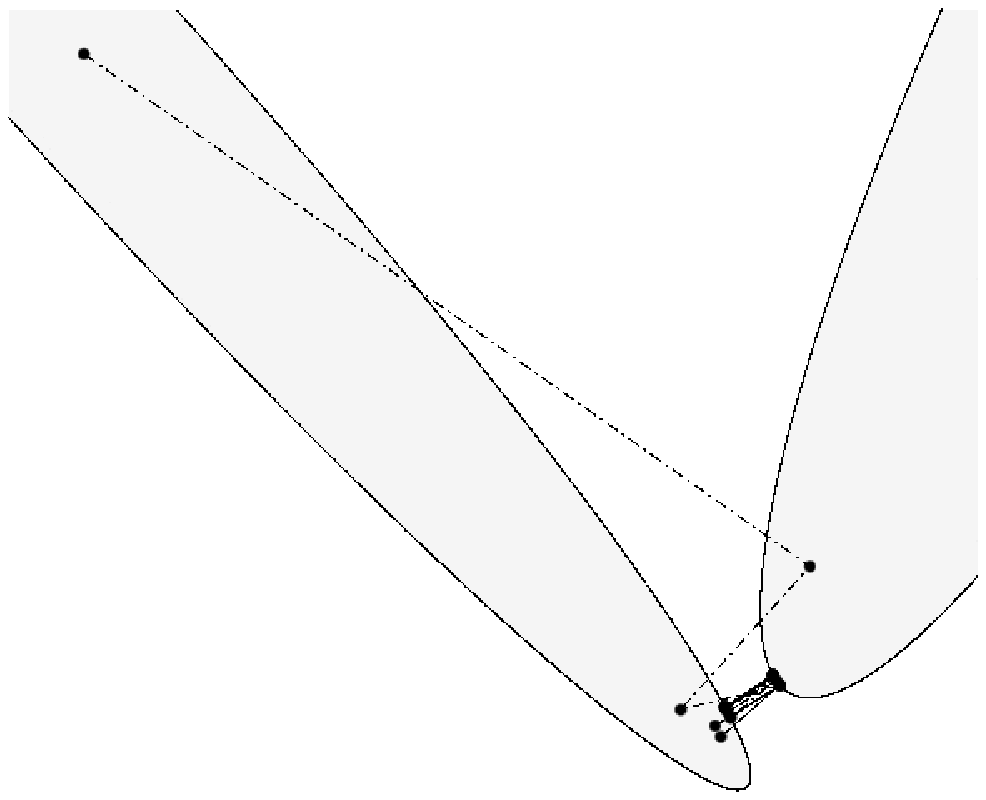}\\
\end{figure}

\end{minipage} \hfill
\begin{minipage}[b]{0.49\linewidth}
\begin{center}
(d) ExactAlg\ref{Alg:ACondG-2} ($[z_0^B]_1=2.50$)
\end{center}\vspace{-0.80cm}
\begin{figure}[H]
	\centering
		\includegraphics[scale=0.50]{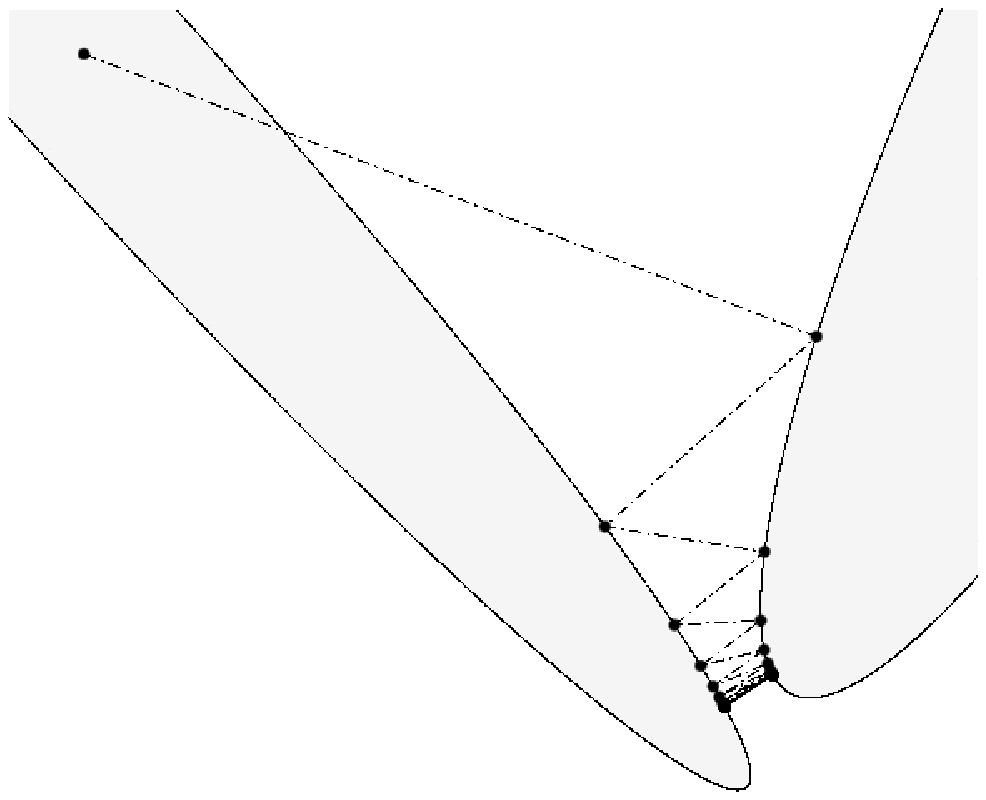}\\
\end{figure}
\end{minipage}\hfill
\caption{Behavior of ACondG-2 and ExactAlg\ref{Alg:ACondG-2} algorithms on the problem of finding a point in $A\cap B$, where the sets are given as $\eqref{setA-1}$ and $\eqref{setB-2}$ with $[z_0^B]_1=2.30$ and $[z_0^B]_1=2.50$.}
\label{fig:ACondG-2}
\end{figure}

In the four feasible instances, ACondG-2 found a point in the intersection of $A$ and $B$ while ExactAlg\ref{Alg:ACondG-2} stopped due to lack of progress.
As for the exact scheme in the previous section, ExactAlg\ref{Alg:ACondG-2} got stuck at infeasible points {\it near} the intersection of the sets, see Figure~\ref{fig:ACondG-2}(b).
We report that, with respect to ExactAlg\ref{Alg:ACondG-2}, the feasibility violation measure $\min\{c_B(x^{k+1}),c_A(y^{k+1})\}$ arrived $\mathcal{O}(10^{-8})$ after  35, 118, and 386 iterations for the first three instances, respectively, and $\mathcal{O}(10^{-7})$ after 526 iterations for the fourth instance. In its turn, as can be seen from Table~\ref{tab:ACondG-2}, ACondG-2 found a feasible point in 2, 2, 8, and 155 iterations, respectively, showing the {\it huge} performance difference between the methods for this class of problem. As suggested by Figure~\ref{fig:ACondG-2}(a), since the iterates lie in the interior of the two sets, ACondG-2 may be able to find a feasible point very quickly.

For the infeasible instances, the feasibility violation measure arrived, respectively, $\mathcal{O}(10^{-4})$, $\mathcal{O}(10^{-3})$, $\mathcal{O}(10^{-2})$, and $\mathcal{O}(10^{-1})$ after 33, 5, 3, and 2 iterations for ACondG-2 and after 82, 15, 5, and 2 for ExactAlg\ref{Alg:ACondG-2}. This shows that, for the chosen set of problems, ACondG-2 approaches the nearest region between $A$ and $B$ faster than ExactAlg\ref{Alg:ACondG-2}, see Figures~\ref{fig:ACondG-2}(c) and (d). On average, ACondG-2 required $65.7\%$ fewer iterations than ExactAlg\ref{Alg:ACondG-2} for stopping due to lack of progress. On the other hand, in some instances, ExactAlg\ref{Alg:ACondG-2} obtained a final iterate with a smaller feasibility violation measure. This can be explained by the fact that algorithms use different numerical approaches to calculate projections. 

Last but not least, the performance of the methods presented throughout the numerical results section should be taken as an illustration of the capabilities of the introduced methods with respect to their exact counterparts, taking into account that they correspond to small problems with specific structures. More precise  conclusions should be made after numerical experiments using problems of different classes and scales.

\section{Conclusions} \label{Sec:Conclusions}
In the present  paper,  we proposed a new method to solve Problem~\eqref{def:InexactProjProb} by combining CondG method  with the  alternate directions method.  As suggested by the numerical experiments, the proposed method seems promising. Let us highlight some aspects observed during the numerical  tests.   In the chosen set of test problems,  the inexact methods performed  fewer iterations than the exact ones. In particular,  whenever the intersection of the involved sets  has a nonempty interior, the methods  converged  in a finite  number of iterations. These phenomena deserve further investigations.  It would also be interesting to extend  ACondG method to the convex feasibility  problem with multiple  involved sets. Finally, the  CondG method   can also  be  used to design inexact versions of several projection methods, including  but not limited to  averaged projections method~\cite{LewisLukeMalick2009}, Han's method~\cite{Han1998} (see also \cite{BauschkeBorwein1994})  or more generally Dykstra's alternating projection method~\cite{BauschkeBorwein1994}. For more variants of projections methods see \cite[Section~III]{Combettes1993}.   For instance, one  inexact version  of averaged projection method for two sets  is stated as follows:
\begin{algorithm} 
\begin{description}
\item[ Step 0.] Let   $(\lambda_k)_{k\in\mathbb{N}}$, $(\gamma_k)_{k\in\mathbb{N}}$,  and $(\theta_k)_{k\in\mathbb{N}}$  be  sequences of nonnegative real numbers and  $\varphi_k:=\varphi_{ \gamma_k, \theta_k, \lambda_k}$, as defined in \eqref{eq:pfphi}.   Let $x_0\in A$, $y_0\in B$,  and set $z_0:=(x_0+y_0)/2$. Initialize $k\leftarrow 0$.
\item[ Step 1.] If   $z^{k}\in  A\cap B$, then {\bf stop}.
\item[ Step 2.] Using   Agorithm \ref{Alg:CondG},  compute   $\mbox{CondG$_{A}$}\,(\varphi_k, x^k,z^{k})$ and $\mbox{CondG$_{B}$}\,(\varphi_k, y^k,  z^k)$ and set the next iterate  $z^{k+1}$ as  
\begin{equation*} 
 z^{k+1}:=\frac{1}{2}\left[ \mbox{CondG$_{A}$}\,(\varphi_k, x^k,z^{k}) + \mbox{CondG$_{B}$}\,(\varphi_k, y^k,  z^k)   \right].  
\end{equation*}  
\item[ Step 3.]  Set $k\gets k+1$, and go to \textbf{Step~1}.
\end{description}
\caption{Averaged projection method with  inexact  projections onto  two sets}
\end{algorithm}


\end{document}